\documentclass[12pt]{amsart}

\usepackage{latexsym, amssymb, amsmath,ulem,soul,esint}

\usepackage{amsfonts, graphicx}
\usepackage{graphicx,color}
\usepackage{url}
\usepackage{graphicx}
\usepackage{epstopdf}

\newcommand{\be}{\begin{equation}}
\newcommand{\ee}{\end{equation}}
\newcommand{\beq}{\begin{eqnarray}}
\newcommand{\eeq}{\end{eqnarray}}

\newtheorem{thm}{Theorem}[section]

\newtheorem{lma}{Lemma}[section]
\newtheorem{prop}{Proposition}[section]

\newtheorem{defn}{Definition}[section]

\newtheorem{claim}{Claim}[section]
\theoremstyle{remark}
\newtheorem{rem}{Remark}[section]
\numberwithin{equation}{section}

\DeclareMathOperator{\dist}{dist}
\DeclareMathOperator{\sign}{sgn}
\DeclareMathOperator{\Int}{Int}

\def\be{\begin{equation}}
\def\ee{\end{equation}}
\def\bee{\begin{equation*}}
\def\eee{\end{equation*}}

\newcommand{\de}{\partial}

\newcommand{\Ric}{\mathrm{Ric}}

\newcommand{\vp}{\varphi}
\newcommand{\e}{\epsilon}
\newcommand{\Vol}{\mathrm{Vol}}

\def\Ric{\text{\rm Ric}}
\def\Rm{\text{\rm Rm}}

\def\a{{\alpha}}

\begin{document}

\title{singular positive mass theorem with arbitrary ends}

\author[J. Chu]{Jianchun Chu$^1$}
\address[Jianchun Chu]{School of Mathematical Sciences, Peking University, Yiheyuan Road 5, Beijing, P.R.China, 100871}
\email{jianchunchu@math.pku.edu.cn}

\author[M.-C. Lee]{Man-Chun Lee$^2$}
\address[Man-Chun Lee]{Department of Mathematics, The Chinese University of Hong Kong, Shatin, N.T., Hong Kong}
\email{mclee@math.cuhk.edu.hk}

\author[J. Zhu]{Jintian Zhu$^3$}
\address[Jintian Zhu]{Beijing International Center for Mathematical Research, Peking University, Beijing, 100871, P. R. China}
\email{zhujt@pku.edu.cn, zhujintian@bicmr.pku.edu.cn}

\thanks{$^1$Research partially supported by Fundamental Research Funds for the Central Universities (No. 7100603624).}

\thanks{$^2$Research partially supported by Hong Kong RGC grant (Early Career Scheme) of Hong Kong No. 24304222 and a direct grant of CUHK}

\thanks{$^3$Research partially supported by China postdoctoral science foundation (grant BX2021013).}

\date{\today}

\begin{abstract}
Motivated by the recent progress on positive mass theorem for asymptotically flat manifolds with arbitrary ends and the Gromov's definition of scalar curvature lower bound for continuous metrics, we start a program on the positive mass theorem for asymptotically flat manifolds with $C^0$ arbitrary ends. In this work as the first step, we establish the positive mass theorem of asymptotically flat manifolds with $C^0$ arbitrary ends when the metric is $W^{1,p}_{\mathrm{loc}}$ for some $p\in(n,\infty]$ and is smooth away from a {\it non-compact} closed subset with Hausdorff dimension $n-\frac{p}{p-1}$. New techniques are developed to deal with non-compactness of the singular set.
\end{abstract}


\maketitle

\section{Introduction}

One of the most fundamental questions in scalar curvature geometry is to ask to what extend the Euclidean space is rigid under nonnegative scalar curvature condition. In \cite{SchoenYau1979C,SchoenYau1981,SchoenYau2017}, Schoen-Yau proved the famous positive mass theorem which says that the Arnowitt-Deser-Misner (ADM) mass of each end of an $n$-dimensional asymptotically flat manifold with nonnegative scalar curvature is nonnegative. Moreover, if the mass of an end vanishes, then the manifold must be isometric to the standard Euclidean space. If the manifold is assumed to be spin, it was also proved by Witten \cite{Witten1981} using spinor method.

Recently, there has been interest in generalizing the positive mass theorem to situation when some of its other ends is only complete but far from being Euclidean. We call these manifolds to be asymptotically flat manifolds with arbitrary ends. In this direction,  it was recently proved by Lesourd-Unger-Yau \cite{LesourdUngerYau2021} that the positive mass theorem for asymptotically Schwarzschild manifolds with arbitrary ends still holds in dimension up to seven.  It was later discovered by the third named author \cite{Zhu2022A} and Lee-Lesourd-Unger \cite{LeeLesourdUnger2022} independently that the approach of Lesourd-Unger-Yau \cite{LesourdUngerYau2021} is indeed sufficient to deal with the general asymptotically flat end by establishing a kind of density theorem.

On the other hand, metrics with low-regularity arise naturally in the compactness theory and in the study of Brown-York quasi-local mass. The positive mass theorem with non-smooth metrics was first studied by Miao \cite{Miao2003} when the singularity is a hypersurface satisfying certain conditions on mean curvature. It was then used by Shi-Tam \cite{ShiTamJDG} to establish the positivity of Brown-York quasi-local mass. Since then, there has been many works in generalizing the result of Miao, see \cite{JiangShengZhang2020,Lee2013,LeeLeFloch2015,LeeTam2021,LiMantoulidis2019,McFeronSzekelyhidi2012,ShiTamPJM} and the references therein. Particularly in \cite{JiangShengZhang2020}, Jiang-Sheng-Zhang considered the situation when the metric is $W^{1,p}_{\mathrm{loc}}$ for $n\leq p\leq \infty$ across a compact singularity set with certain assumptions on its Hausdorff measure, see also \cite{ShiTamPJM} for the earlier works in this direction. At the same time, it is also interesting to compare it with Gromov's program on defining the scalar curvature lower bound for continuous metric based on comparison theorems in scalar curvature geometry. We refer interested readers to Gromov's lecture note \cite{GromovLecture} on scalar curvature for a comprehensive overview.

Motivated by these, we study the positive mass theorem of asymptotically flat manifolds with non-smooth arbitrary ends. To clarify the notations, we begin with the definitions. We point out that all manifolds are assumed to be orientable in the rest of this paper.
\begin{defn}
Let $M^n$ be a non-compact manifold without boundary and $g$ be a complete $C^0$ metric on $M$. A distinguished end $\mathcal{E}$ of $M$ is said to be asymptotically flat if
\begin{enumerate}\setlength{\itemsep}{1mm}
\item[(i)] $\mathcal{E}$ is diffeomorphic to $\mathbb{R}^n\setminus \overline{ B_{\mathrm{euc}}(1)}$;
\item[(ii)] $g$ restricted on $\mathcal{E}$ is smooth such that
\begin{equation*}
|g_{ij}-\delta_{ij}|+r|\partial g_{ij}|+r^2 |\partial^2 g_{ij}|\leq Cr^{-\sigma},\; r=|x|;
\end{equation*}
for some $\sigma>\frac{n-2}{2}$.
\item[(iii)] the scalar curvature $R(g)$ of $g$ defined on $\mathcal{E}$ is in $L^1(\mathcal{E})$.
\end{enumerate}
We say that $(M^n,g,\mathcal{E})$ is an asymptotically flat manifold with $C^0$ arbitrary ends.
\end{defn}

On the asymptotically flat end $\mathcal{E}$, the ADM mass associated to it is defined to be
\begin{defn}[\cite{ArnowittDeserMisner1961}]
The ADM mass of $\mathcal{E}$ is defined by
\begin{equation}
m(M,g,\mathcal{E}) = \lim_{r\to +\infty} \frac1{2(n-1)|\mathbb{S}^{n-1}|} \int_{\{|x|=r\}} (g_{ij,j}-g_{jj,i}) \nu^i \, dx
\end{equation}
where $g_{ij,k}=\de_{k}g_{ij}$ denotes the ordinary partial derivative, $|\mathbb{S}^{n-1}|$ is the volume of unit sphere and $\nu$ is the outward unit normal of $\{x\in \mathbb{R}: |x|=r\}$ in $\mathbb{R}^n$.

\end{defn}

By the work of Bartnik \cite{Bartnik1986}, it is known that the ADM mass is a geometric quantity independent of the choice of coordinate system at infinity.

\medskip

Motivated by the works of the third name author \cite{Zhu2022A} and Lee-Lesourd-Unger \cite{LeeLesourdUnger2022} on the positive mass theorem with arbitrary ends and the removable singularity theorem of Jiang-Sheng-Zhang \cite{JiangShengZhang2020}, we study the problem of singular positive mass theorem  with arbitrary ends when the singularity is possibly non-compact appearing on the other ends. The following is the main result.

\begin{thm}\label{intro-main}
Let $3\leq n\leq 7$ and $(M^n,g,\mathcal{E})$ be an asymptotically flat manifold with $C^0$ arbitrary ends. Assume that
\begin{enumerate}\setlength{\itemsep}{2mm}
\item[(i)] the metric $g$ is in $W^{1,p}_{\mathrm{loc}}(M)\cap C^\infty(M-S)$, where $S$ is a closed (not necessarily bounded) subset in $M$ disjoint from $\mathcal E$ with either
\begin{enumerate}\setlength{\itemsep}{1mm}
    \item $\mathcal{H}_{\mathrm{loc}}^{n-\frac{p}{p-1}}(S)<\infty$ if $n<p<\infty$ or
    \item  $\mathcal{H}^{n-1}(S)=0$ if $p=\infty$.
\end{enumerate}
\item[(ii)] $R(g)\geq 0$ in $M-S$.
\end{enumerate}
Then we have $m(M,g,\mathcal{E})\geq 0$.
\end{thm}

We also prove the following rigidity result for $p\in (n,\infty)$ and $p=\infty$ under an additional assumption. We expect that the additional a-priori regularity of $g$ in case of $p=\infty$ is unnecessary.

\begin{thm}\label{intro-rigidity}
Under the same assumptions of Theorem \ref{intro-main}, assume that either
\begin{enumerate}\setlength{\itemsep}{1mm}
\item[(a)] $n<p<\infty$ or
\item[(b)] $g$ is $L^\infty(M)\cap W^{1,\infty}(M)$ with respect to a smooth metric $h$ of bounded geometry of infinity order (see Definition \ref{Defn: 4.1} and \ref{Defn: 4.2}).
\end{enumerate}
If $m(M,g,\mathcal{E})=0$, then $(M,g)$ is isometric to the standard Euclidean space as a metric space. Moreover, there is a $C^{1,\a}_{\mathrm{loc}}$ diffeomorphism $\Psi:M\to \mathbb{R}^n$ for some $\a\in (0,1)$ which is smooth outside $S$ such that $g=\Psi^*g_{\mathrm{euc}}$ on $M$.
\end{thm}

In contrast with the previous works, the main distinction of our work here is that we allow the singular set $S$ to be unbounded, which leads to great trouble when we try to reduce the problem to the smooth case. Indeed, there are two reasons why the standard smoothing-and-conformal argument fails in our setting. Firstly, from the unboundness of $S$, the smoothing metric has negative scalar curvature in an unbounded region, then the solvability of the corresponding conformal equation is in doubt since we do not have a global Sobolev inequality due to unknown geometry on those arbitrary ends. Secondly, even if the conformal equation can be solved, it is hard to obtain the completeness of the conformal metric since we cannot obtain enough control on the decay rate of the conformal factor. We point out that the positive lower bound of the conformal factor can be obtained through the maximum principle \cite{LeeLesourdUnger2022} or a shifting argument \cite{Zhu2022A} when $S$ is bounded (and so the conformal factor is harmonic outside a compact subset). However those arguments both fail when $S$ happens to be unbounded.

Our strategy here is largely inspired by the shielding theorem obtained in \cite{LesourdUngerYau2021} where it was shown that the geometry far behind a positive scalar curvature region does not affect the positivity of the ADM mass. It is well-known that negative mass can be transformed into positive scalar curvature inside through some deformation, and this further indicates that we can cut away those arbitrary ends at an appropriate position to prove the desired positive mass theorem. This is indeed the case and let us sketch our proof as follows.
Suppose on the contrary that the asymptotically flat manifold with $C^0$ arbitrary ends in Theorem \ref{intro-main} has negative mass. At the first step, we want to transform the negative mass to positive scalar curvature inside. In practice, we establish a singular density theorem (i.e.\ Theorem \ref{Prop: density}) such that the distinguished end becomes conformally flat with negative mass while the singularity keeps unchanged. Next, we can use the compactification method of Lohkamp \cite{Lohkamp1999} to close the manifold by $n$-torus $\mathbb{T}^{n}$ and now we obtain a singular metric on $\mathbb T^n\# N$ for some non-compact manifold $N$, whose scalar curvature is nonnegative in the smooth part and is positive in a neck region. Then we establish a shielding version of the Geroch conjecture (i.e.\ Proposition \ref{Thm: Geroch conjecture}) on $\mathbb T^n\#N$ using the method of $\mu$-bubbles, which guarantees us to find a right position $\mathbb T^n\# N$ to cut away arbitrary ends such that the remaining compact parts (still in the form of $\mathbb T^n\#\tilde N$ for some compact $\tilde N$) still violate our shielding version of the Geroch conjecture. Now we return to a singular problem on compact manifolds and the standard smoothing argument works to obtain a contradiction. (We find that the warping trick from \cite{FischerColbrieSchoen1980} is more convenient than conformal deformation to show the stability of the shielding condition \eqref{Eq: shielding condition}, see Proposition \ref{Prop: smoothing}.)

\medskip

The organization of this paper is as follows. In Section 2, we will establish a shielding version of the Geroch conjecture. In Section 3, we will establish the density theorem and use it to prove the positive mass theorem. In Section 4, we will discuss the case when the ADM mass of the distinguished end vanishes.

\section{A shielding version of Geroch conjecture}
The shielding principle for Riemannian manifolds with nonnegative scalar curvature yields that boundary far behind a domain with positive scalar curvature behaves like mean-convex. Such philosophy has appeared in many works, for instances see  \cite{CecchiniZeidler2021,LeeLesourdUnger2022}.  In this work, we consider the Georoch conjecture using the same perspective. Recall that the resolution of the classical Geroch conjecture says that the $n$-torus $\mathbb{T}^n$ admits no smooth metric with positive scalar curvature. This conjecture was proved independently by Gromov-Lawson \cite{GromovLawson1980} and Schoen-Yau \cite{SchoenYau1979A,SchoenYau1979B,SchoenYau2017}.  In fact, Schoen-Yau \cite{SchoenYau2017} are able to rule out the existence of positive scalar curvature metrics on the connected sum $\mathbb{T}^n\# X$, where $X$ is an arbitrary closed orientable $n$-manifold.  More recently, Chodosh-Li \cite{ChodoshLi2020} showed for $n\leq 7$ that there is no complete smooth metric on $\mathbb{T}^n\#X$ with positive scalar curvature even when $X$ is an orientable open $n$-manifold.  When $X$ is a compact orientable $n$-manifold with boundary, it is not difficult to see that the minimal surface argument by Schoen-Yau \cite{SchoenYau2017} also asserts that $\mathbb{T}^n\#X$ admits no smooth metric with positive scalar curvature and mean convex boundary. Motivated by this, in this work we consider the shielding principle under the setting with boundary. Indeed, we shall consider the slightly more general curvature condition, the $\mathbb{T}^l$-stablized scalar curvature lower bound which is introduced by Gromov in \cite{Gromov2020}.
\begin{defn}
Let $(N,g)$ be a Riemannian manifold and $\sigma:N\to \mathbb R$ a continuous function on $N$. We say that $(N,g)$ has {\it $\mathbb{T}^l$-stablized scalar curvature lower bound $\sigma$} if there are $l$ positive smooth functions $u_1,u_2,\ldots,u_l$ such that the warped metric
\begin{equation}
g_{\mathrm{warp}}=g+\sum_{i=1}^lu_i^2 d\theta_i^2
\end{equation}
on $N\times \mathbb{T}^l$ satisfies $R(g_{\mathrm{warp}})\geq \sigma$.
\end{defn}

The main result in this section is as follows.
\begin{thm}\label{Thm: Geroch conjecture}
Let $n\leq 7$. Assume $(N^n,g)$ is a compact Riemannian manifold with boundary associated with a finite open exhaustion
\begin{equation}
\mathbb{T}^n-B\approx U_0\subset U_1\subset U_2\subset N.
\end{equation}
Assume that there is a nonnegative continuous function $\sigma$ such that $(N,g)$ has $\mathbb{T}^l$-stablized scalar curvature lower bound $\sigma$.
Then we have
\begin{equation}
\inf_{\bar{\mathcal A}}\sigma\cdot D_0 \cdot D_1\leq 4,
\end{equation}
where $\mathcal A=U_2-\bar U_1$, $D_0=\dist(\partial U_2,U_1)$ and $D_1=\dist(\partial N,U_2)$.
\end{thm}
\begin{rem}
Without much effort we can obtain the same conclusion when $\mathbb{T}^n$ is replaced by an arbitrary Schoen-Yau-Schick manifold (see Definition \ref{Defn: SYS} and \cite[Section 5]{Gromov2018}).
\end{rem}

The proof of Theorem~\ref{Thm: Geroch conjecture} is based on using $\mu$-bubble and topology to derive contradiction. We start with introducing the Schoen-Yau-Schick manifolds.

\begin{defn}[Schoen-Yau-Schick (SYS) manifolds]\label{Defn: SYS}
A closed orientable $n$-manifold $N$ is said to be a Schoen-Yau-Schick  manifold if there exist $(n-2)$ cohomology $1$-class $\beta_1,\ldots,\beta_{n-2}$ in $H^1(N,\mathbb Z)$ such that
\begin{equation}
[N]\frown(\beta_1\smile\beta_2\smile\cdots\smile \beta_{n-2})\in H_2(N,\mathbb Z)
\end{equation}
is non-spherical, i.e. any smooth representation cannot consist of $2$-spheres.
\end{defn}

\begin{rem}\label{Rem: SYS}
If $N$ admits a non-zero degree map to $\mathbb{T}^n$, then it is a Schoen-Yau-Schick manifold.
\end{rem}

The following is a generalization of Schoen-Yau's non-existence theorem on torus \cite{SchoenYau1979B,SchoenYau2017}.

\begin{prop}\label{Prop: Tn no PSC}
Let $n\leq 7$ and $N^n$ is a SYS manifold. Then $N$ cannot admit any smooth metric with $\mathbb{T}^l$-stablized scalar curvature lower bound $\sigma>0$.
\end{prop}

\begin{proof}
This follows easily from the dimension descent argument by Schoen-Yau (see \cite{SchoenYau2017} for instance).  We include a sketch for the reader's convenience. Suppose on the contrary, there is a smooth metric $g$ on $N$ and positive smooth functions $u_1, u_2,\ldots, u_l$ such that the warped metric
\begin{equation}
g_{\mathrm{warp}}=g+\sum_{i=1}^lu_i^2 dt_i^2
\end{equation}
on $N\times \mathbb{T}^l$ satisfies $R(g_{\mathrm{warp}})\geq \sigma>0$.

Consider the following minimization problem.  Let
\begin{equation}
\mathcal A(\Sigma)=\mathcal H^{n+l-1}_{g_{\mathrm{warp}}}(\Sigma\times \mathbb{T}^l)=\int_{\Sigma} \rho\, d\mu_g
\end{equation}
among all hypersurfaces $\Sigma$ in the homology class $[N]\frown \beta_1$ where $\rho=\prod_{i=1}^l u_i$. Since $n\leq 7$, by geometric measure theory we can find a smooth hypersurface $\Sigma_1$ representing the homology class $[N]\frown \beta_1$ such that it minimizes the functional $\mathcal A$.  Moreover, since $u_i$ are $\mathbb{T}^l$-invariant, it follows from the second variation formula and the Gauss equation (for instances see \cite[Lemma 14]{ChodoshLi2020}) that for all $f\in C^\infty(\Sigma_1)$ and extend it trivially to $\Sigma_1\times \mathbb{T}^l$, we have
\begin{equation}
0 \leq  \int_{\Sigma_1\times \mathbb{T}^l}-f\Delta_{\Sigma_1\times \mathbb{T}^l}f -\frac12(|h_{\Sigma_1}|^2+\sigma-R_{\Sigma_1\times \mathbb{T}^l})f^2 \;d\mu_g,
\end{equation}
where $\Delta_{\Sigma_1\times \mathbb{T}^l}$ denotes the Laplacian of the metric on $\Sigma_1\times \mathbb{T}^l$ induced by the warped product metric and $R_{\Sigma_1\times \mathbb{T}^l}$ denotes the scalar curvature of the induced metric. This shows that there exists a positive function $u_{l+1}\in C^\infty(\Sigma_1)$ such that
\begin{equation}
\Delta_{\Sigma_1\times \mathbb{T}^l} u_{l+1}\leq \frac12(-\sigma+R_{\Sigma_1\times \mathbb{T}^l})u_{l+1}.
\end{equation}
Since $u_i$ are $\mathbb{T}^l$-invariant, we see that the metric $\pi^* g +\sum_{i=1}^{l+1}u_i^2 dt^2_i$ on $\Sigma_1\times \mathbb{T}^{l+1}$ has scalar curvature
\begin{equation}
R_{\Sigma_{1}\times T^{l}}-\frac{2\Delta_{\Sigma_1\times \mathbb{T}^l} u_{l+1}}{u_{l+1}} \geq R_{\Sigma_{1}\times T^{l}}+\sigma-R_{\Sigma_{1}\times T^{l}} = \sigma.
\end{equation}
Equivalently, $\Sigma_1$ has $\mathbb{T}^{l+1}$-stabilized scalar curvature lower bound $\sigma>0$.

By the assumption
\begin{equation}
[\Sigma_1]\frown (\beta_2\smile\cdots\smile \beta_{n-2})\mbox{ is non-spherical},
\end{equation}
we can apply the above argument inductively to obtain a non-spherical closed orientable surface $\Sigma_{n-2}$ which has $T^{l+n-2}$-stabilized scalar curvature lower bound $\sigma>0$. Let $\Sigma_{n-2}'$ be one of the component of $\Sigma_{n-2}$ with positive genus. From the proof of \cite[Lemma 2.3]{Zhu2020A} (see also \cite[Lemma 2.5]{SchoenYau2017}), we have
\begin{equation}
0<\int_{\Sigma_{n-2}'}\sigma\,\mathrm dA\leq 8\pi\chi(\Sigma_{n-2}')\leq 0,
\end{equation}
which leads to a contradiction.
\end{proof}

Next, we will need to construct the weight for the $\mu$-bubble. The next lemma is a slight modification of \cite[Lemma 2.3]{Zhu2020B}, which can be proved by an almost identical argument.

\begin{lma}\label{Lem: h epsilon}
For any $\e\in(0,1)$, there is a function
\begin{equation}
h_\e:\left(-\e^{-1},\e^{-1} \right)\to\mathbb R
\end{equation}
such that
\begin{enumerate}
\item[(i)] $h_\e$ satisfies
\begin{equation*}
h_\e^2+2h_\e'=\e^2 \mbox{ on } (-\e^{-1},-1]\cup[1, \e^{-1})
\end{equation*}
and there is a universal constant $C$ independent of $\e$ so that
\begin{equation*}
\sup_{[-1,1]} \left|h_\e^2+2h_\e'\right|\leq C\e;
\end{equation*}
\item[(ii)] $h_\e'<0$ and
\begin{equation*}
\lim_{t\to\mp \e^{-1}} h_\e(t)=\pm\infty;
\end{equation*}
\item[(iii)] $h_\e$ converge smoothly to $0$ on any closed interval as $\e\to 0$.
\end{enumerate}
\end{lma}

We also need the construction of another weight function.

\begin{lma}\label{Lem: f}
Given positive constants $\sigma_0$, $t_0$ and $\e$, we can find a smooth function $f:[0,T)\to (-\infty,0]$ with
\begin{equation}
T\leq t_0+\frac{4+\e}{\sigma_0 t_0}
\end{equation}
such that
\begin{enumerate} \setlength\itemsep{1mm}
\item[(i)] $f$ vanishes around $t=0$ and $f$ diverges to $-\infty$ as $t\to T$;
\item [(ii)]  $f$ is monotone non-increasing;
\item[(iii)] If $T\leq t_0$, then $f$ satisfies
\begin{equation}\label{Eq: Sc condition 1}
f^2+2f'\geq -\sigma_0\mbox{ in }[0,T).
\end{equation}
Otherwise $f$ satisfies
\begin{equation}\label{Eq: Sc condition 2}
f^2+2f'\geq -\sigma_0\mbox{ in }[0,t_0]\mbox{ and }f^2+2f'\geq 0 \mbox{ in }[t_0,T).
\end{equation}
\end{enumerate}
\end{lma}
\begin{proof}
Fix a small positive constant $\delta$. We take a cutoff function
\begin{equation}
\eta:[0,+\infty)\to [0,1]
\end{equation}
supported in $[0,t_0]$ such that $\eta\equiv 0$ around $t=0$ and when $t\geq t_0$, and also $\eta\equiv 1$ in $[\delta t_0,(1-\delta)t_0]$. Let $f$ be the solution to the following ordinary differential equation
\begin{equation}
f^2+2f'=-\sigma_0\eta\mbox{ in }(0,\infty)\mbox{ with }f(0)=0.
\end{equation}
Clearly the function $f$ is monotone non-increasing and it vanishes around $t=0$ and satisfies \eqref{Eq: Sc condition 1} or \eqref{Eq: Sc condition 2} depending on the value of $T$. Now let us estimate the blow-up moment $T$. Notice that we have
\begin{equation}
-\frac{1}{2}\mathrm dt=\frac{\mathrm df}{f^2+\sigma_0}\mbox{ when }\delta t_0\leq t\leq (1-\delta)t_0.
\end{equation}
If $T\leq t_0$, we are done. So we just work with the possibility where $T>t_0$. Through integration from $\delta t_0$ to $(1-\delta)t_0$ we obtain
\begin{equation}
\arctan\left(\frac{f(t_0-\delta t_0)}{\sqrt {\sigma_0}}\right)=-\frac{\sqrt{\sigma_0}}{2}(t_0-2\delta t_0)+\arctan\left(\frac{f(\delta t_0)}{\sqrt {\sigma_0}}\right).
\end{equation}
From monotonicity of $f$ we see $f(\delta t_0)\leq 0$ and also
\begin{equation}
f(t_0)\leq f(t_0-\delta t_0)\leq -\sqrt{\sigma_0}\tan\left(\frac{\sqrt{\sigma_0}}{2}(t_0-2\delta t_0)\right).
\end{equation}
From a similar analysis on interval $[t_0,t]$ with $t<T$ we obtain
\begin{equation}
0 > f(t)=\left(\frac{1}{f(t_0)}+\frac{1}{2}(t-t_0)\right)^{-1}.
\end{equation}
This implies
\begin{equation}
\begin{split}
t-t_0&<\left(\frac{\sqrt{\sigma_0}}{2}\tan\left(\frac{\sqrt{\sigma_0}}{2}(t_0-2\delta t_0)\right)\right)^{-1}= \frac{G(\sqrt{\sigma_0}t_0)}{\sigma_0 t_0},
\end{split}
\end{equation}
where
\begin{equation}
G(s)=\frac{2s}{\tan\left(\frac{1-2\delta}{2}s\right)}\leq \frac{4}{1-2\delta}.
\end{equation}
After picking up sufficiently small $\delta$ we have $T\leq t_0+(4+\e)(\sigma_0 t_0)^{-1}$.
\end{proof}

Now we are ready to prove Theorem \ref{Thm: Geroch conjecture}.
\begin{proof}[Proof of Theorem \ref{Thm: Geroch conjecture}]

Denote $E=N-U_0$. By collapsing $E$ to a point and taking a suitable smooth flat product metric $g_0$ on $\mathbb{T}^n$, we can construct a smooth map
\begin{equation}
F:(N,E,g)\to (\mathbb{T}^n,p,g_0)
\end{equation}
such that $F$ is a diffeomorphism between $U_0$ and $\mathbb{T}^n-\{p\}$ and its Lipschitz norm is less than or equal to one. Take the covering space $\pi_{0}:\mathbb{T}^{n-1}\times \mathbb{R}\to \mathbb{T}^n$ and denote $\pi_{0}^{-1}(p)=\bigcup_{i=1}^\infty\{p_i\}$. By considering $\mathbb{T}^{n-1}\times  \mathbb{R}$ as a principal $\mathbb Z$-bundle over $\mathbb{T}^n$, we can pull it back over $N$ to obtain a covering space $\pi:\tilde N\to N$ given by
\begin{equation}
\tilde N=\{(x,y)\in N\times (\mathbb{T}^{n-1}\times \mathbb{R})\,|\,F(x)=\pi_{0}(y)\}
\end{equation}
such that $\pi^{-1}(E)$ has countably many components and each component $E_i$ corresponds to a point $p_i$. Moreover the map $F$ can be lifted to a proper map $\tilde F:\tilde N\to \mathbb{T}^{n-1}\times \mathbb{R}$ such that $\tilde F(E_i)=\{p_i\}$ and $\tilde F$ induces a diffeomorphism between $\tilde N-\bigcup_{i=1}^\infty E_i$ and $\mathbb{T}^{n-1}\times \mathbb{R}-\bigcup_{i=1}^\infty\{p_i\}$. Take one $\mathbb{T}^{n-1}$-slice disjoint from all points $p_i$ and then $\Sigma_0=\tilde F^{-1}(\mathbb{T}^{n-1})$ is an embedded hypersurface in $\tilde N$. Denote $\rho_1:\mathbb{T}^{n-1}\times \mathbb{R}\to \mathbb{R}$ to be the projection map. After changing the $\mathbb{R}$-parameter we can assume that this $\mathbb{T}^{n-1}$-slice is $\mathbb{T}^{n-1}\times \{0\}$ and the Lipschitz norm of $\rho_1$ is less than or equal to one.

It suffices to prove Theorem \ref{Thm: Geroch conjecture} when $\inf_{\bar{\mathcal A}}\sigma$, $D_0$ and $D_1$ are all positive constants. Suppose on the contrary,  we have $\inf_{\bar{\mathcal A}}\sigma\cdot D_0\cdot D_1>4$. Since $N$ has $\mathbb{T}^l$-stabilized scalar curvature lower bound $\sigma$, we can find smooth positive functions $u_1, u_2,\ldots, u_l$ such that the warped metric
\begin{equation}
g_{\mathrm{warp}}=g+\sum_{i=1}^l u_i^2 dt_i^2
\end{equation}
on $N\times \mathbb{T}^l$ satisfies $R(g_{\mathrm{warp}})\geq \sigma$.  Instead of considering the metric and curvature on the warped product manifold, we treat it as a function on $N$ since $u_i$ are $\mathbb{T}^l$-invariant.

Fix a small positive constant $\e$ and consider the first eigenfunction $u_{l+1}$ such that
\begin{equation}
\left\{
\begin{array}{ll}
-\Delta_gu_{l+1}+\e R(g_{\mathrm{warp}})u_{l+1}=\lambda u_{l+1} \ \mbox{ in }N;\\[2mm]
\displaystyle \frac{\partial u_{l+1}}{\partial\nu}=0 \ \mbox{ on }\partial N
\end{array}
\right.
\end{equation}
for $\lambda>0$ due to quasi-positivity of $R(g_{\mathrm{warp}})$, and $\nu$ is the outward unit normal of $\partial N$ in $N$.
Denote
\begin{equation}
\bar g_{\mathrm{warp}}=g_{\mathrm{warp}}+u_{l+1}^2 dt_{l+1}^2.
\end{equation}
Direct computation shows that
\begin{equation}
R(\bar g_{\mathrm{warp}})=R(g_{\mathrm{warp}})-\frac{2\Delta_g u_{l+1}}{u_{l+1}}=(1-2\e)R(g_{\mathrm{warp}})+2\lambda.
\end{equation}
Denote $\tilde g_{\mathrm{warp}}$ to be the lifting metric of $\bar g_{\mathrm{warp}}$ on $\tilde N$. It is easy to see that $\tilde N$ has $\mathbb{T}^{l+1}$-stabilized scalar curvature lower bound $(1-2\e)\sigma\circ \pi+2\lambda$.

In the following, we will construct a hypersurface with prescribing mean curvature where we can obtain a contradiction between its topology and the above positive $\mathbb{T}^{l+1}$-stabilized scalar curvature lower bound of $\tilde N$.

We first construct an appropriate smooth function on $\tilde N$ serving as prescribing mean curvature functions.  For the given $\lambda>0$,  we choose $c$ sufficiently large so that  Lemma \ref{Lem: h epsilon} applied to obtain a function $h_1:(-c,c)\to \mathbb{R}$
such that
\begin{enumerate}\setlength{\itemsep}{1mm}
\item[(a)] $h_1(t)\to \pm\infty$ as $t\to\mp c$;
\item[(b)] $h_1'(t)\leq 0$ for all $t$;
\item[(c)] $h_1^2+2h_1'\geq -\lambda$ for all $t$.
\end{enumerate}
Clearly, we might perturb $c$ slightly if necessary so that $\mathbb{T}^{n-1}\times\{\pm c\}$ is disjoint from all points $p_i$ and so $\tilde V=\tilde F^{-1}([-c,c])$ is a smooth compact region in $\tilde N$. Denote
\begin{equation}
\tilde h_1=h_1\circ \rho_1\circ \tilde F:\tilde V\to [-\infty,+\infty].
\end{equation}
Then
\begin{equation}\label{h1-estimate}
    \tilde h_1^2-2|\mathrm d\tilde h_1|\geq -\lambda
\end{equation}
and $\tilde h_1=\pm\infty$ on $\partial_\mp \tilde V:=\tilde F^{-1}(\mp c)$, since $|d\rho_1|\leq 1$.  Moreover, $\tilde h_1$ is a non-zero constant on each $E_i$ contained in $\tilde V$.

Now we need to construct another function $\tilde h$ on $\tilde V$. By mollifying the distance function, it is not difficult to construct a smooth function
\begin{equation}
\rho_2:N\to [0,+\infty)
\end{equation}
such that
\begin{enumerate}\setlength{\itemsep}{1mm}
\item[(i)] $\rho_2\equiv 0$ in $U_1$;
\item[(ii)]$\rho_2\geq (1-\e)D_0$ outside $U_2$;
\item[(iii)] $ \rho_2\geq (1-\e)(D_0+D_1)$ on $\partial N$;
\item [(iv)] $|d\rho_2|\leq 1$ in $N$.
\end{enumerate}

Let us now apply Lemma \ref{Lem: f} with $\sigma_0=(1-2\e)\inf_{\bar{\mathcal A}}\sigma$, $t_0=(1-\e)D_0$ and $\e$ given from above so that we obtain a function $h_2:[0,T)\to (-\infty,0]$ satisfying
\begin{equation}\label{h2-estimate-1}
h_2^2+2h_2'\geq -(1-2\e)\inf_{\bar{\mathcal A}}\sigma\mbox{ in }[0,T),\,\mbox{ if }T\leq t_0,
\end{equation}
or
\begin{equation}\label{h2-estimate-2}
h_2^2+2h_2'\geq -(1-2\e)\inf_{\bar{\mathcal A}}\sigma\mbox{ in }[0,t_0],\, h_2^2+2h_2'\geq 0 \mbox{ in }[t_0,T),\, \mbox{ if }T>t_0,
\end{equation}
and $h_2(t)\to -\infty$ as $t\to T$. Using $\inf_{\bar{\mathcal A}}\sigma\cdot D_0\cdot D_1>4$, for $\e$ small enough, we have
\begin{equation}\label{EQ: upper bound of T}
\begin{split}
T \leq {} & (1-\e)D_0+\frac{1+\e/4}{(1-2\e)(1-\e)}\cdot \frac{4}{\inf_{\bar{\mathcal A}}\sigma\cdot D_0\cdot D_1}\cdot D_1 \\[1mm]
< {} & (1-\e)(D_0+D_1).
\end{split}
\end{equation}
Note that $T$ may not be a regular value of $\rho_2$, but it can be a regular value of $\tau\rho_2$ for almost every $\tau>0$. Since all properties of $\rho_2$ are kept if  we take $\tau>1$ sufficiently closed to one, we do not bother the reader with technical discussions due to introducing constant $\tau$ but just assume that $T$ is a regular value of $\rho_2$. Denote
\begin{equation}
\tilde V_1=\{\tilde x\in \tilde V: \rho_2\circ \pi\leq T\}
\end{equation}
and clearly $\tilde V_1$ is a compact smooth region of $\tilde N$ which is away from $\partial N$ thanks to the upper bound of $T$ \eqref{EQ: upper bound of T}. Let us define
\begin{equation}
\tilde h(\tilde x)=\tilde h_1(\tilde x)+\sign (-\tilde h_1(\tilde x))\cdot h_2(\rho_2(\pi(\tilde x)))\text{ for all }\tilde x\in \tilde V_1,
\end{equation}
where $\sign(\cdot)$ is the sign function. Note that $\tilde h_1(\tilde x)$ and $\sign (-\tilde h_1(\tilde x))\cdot h_2(\rho_2(\pi(\tilde x)))$ must be of same sign.   It is not difficult to verify that $\tilde h$ is a smooth function in the interior of $\tilde V_1$ and $\tilde h$ takes $+\infty$ or $-\infty$ on the boundary of $\tilde V_1$. Moreover, since $h_i$ are non-increasing, we have
\begin{equation}
\tilde h^2-2|\mathrm d\tilde h|+(1-2\e)\sigma\circ \pi+2\lambda\geq \lambda\mbox{ in }\tilde V_1.
\end{equation}
Here we have used \eqref{h1-estimate}, \eqref{h2-estimate-1} and \eqref{h2-estimate-2}.

Next we are ready to construct the desired hypersurface with prescribed mean curvature in $\tilde V_1$. Recall that the warped metric
\begin{equation}
\tilde g_{\mathrm{warp}}=\pi^*g+\sum_{i=1}^{l+1}u_i^2 dt_i^2
\end{equation}
on $\tilde V_1\times \mathbb{T}^{l+1}$ satisfies $R(\tilde g_{\mathrm{warp}})\geq (1-2\e)\sigma\circ \pi+2\lambda$. Denote $\Int \tilde V_1$ to be the interior part of $\tilde V_1$ and take
\begin{equation}
\Omega_0:=\{\tilde x\in\Int\tilde V_1:\tilde h(x)<0\}.
\end{equation}
In particular we have
\begin{equation}
\partial\Omega_0=\{\tilde h(x)=0\}=\{\tilde h_1(x)=0\}=\tilde F^{-1}(\mathbb{T}^{n-1}\times\{0\})=\Sigma_0.
\end{equation}
Among the set
\begin{equation}
\mathcal C=\{\text{Caccioppoli sets $\Omega$ in $\Int\tilde V_1$ such that $\Omega\Delta \Omega_0\Subset \Int \tilde V_1$}\},
\end{equation}
we try to minimize the functional
\begin{equation}
\mathcal A(\Omega)=\mathcal H^{l+n}_{\tilde g_{\mathrm{warp}}}(\partial\Omega\times \mathbb{T}^{l+1})-\int_{\Int\tilde V_1\times \mathbb{T}^{l+1}}(\chi_{\Omega\times \mathbb{T}^{l+1}}-\chi_{\Omega_0\times \mathbb{T}^{l+1}})\tilde h\,\mathrm d\mathcal H^{l+1+n}_{\tilde g_{\mathrm{warp}}}.
\end{equation}
Since the prescribed mean curvature function $\tilde h$ blows up on $\partial\tilde V_1$, we have a nice barrier condition around $\partial\tilde V_1$ and it follows from geometric measure theory that we can find a smooth region $\Omega$ in $\mathcal C$ minimizing the functional $\mathcal A$ (no singularity issue involves here due to $\mathbb{T}^{l+1}$-invariance and the fact $n\leq 7$). It is clear that the boundary $\Sigma=\partial\Omega$ is homologous to $\Sigma_0$ and from a similar argument as in the proof of \cite[Lemma 2.7]{Zhu2022B} it has $\mathbb{T}^{l+2}$-stablized scalar curvature lower bound
\begin{equation}
(1-2\e)\sigma+2\lambda+\frac{n+l+1}{n+l}\tilde h^2-2|\mathrm d\tilde h|\geq \lambda > 0.
\end{equation}
Notice that the composed map $\Sigma\hookrightarrow \tilde N\to \mathbb{T}^{n-1}\times \mathbb{R}\to \mathbb{T}^{n-1}$ has non-zero degree since the map $\Sigma_0\to \mathbb{T}^{n-1}$ does. We obtain a contradiction to Proposition \ref{Prop: Tn no PSC} and Remark \ref{Rem: SYS}.
\end{proof}

\section{Positive mass theorem with singularity: nonnegative mass}

In this section, we will prove the positive mass theorem. We restate the statement for reader's convenience. 
\begin{thm}[i.e.\ Theorem \ref{intro-main}]\label{Thm: nonnegative mass}
Let $3\leq n\leq 7$ and $(M^n,g,\mathcal E)$ be an asymptotically flat manifold with $C^0$ arbitrary ends. Assume that
\begin{itemize}\setlength{\itemsep}{2mm}
\item[(c1)] the metric $g$ is in $W^{1,p}_{\mathrm{loc}}(M)\cap C^\infty(M-S)$, where $S$ is a closed subset in $M$ disjoint from $\mathcal E$ with
\begin{enumerate}\setlength{\itemsep}{1mm}
    \item [(i)]
$\mathcal{H}_{\mathrm{loc}}^{n-\frac{p}{p-1}}(S)<\infty$ if $n<p<\infty$;
\item [(ii)] $\mathcal{H}^{n-1}(S)=0$ if $p=\infty$;
\end{enumerate}
\item[(c2)] $R(g)\geq 0$ in $M-S$.
\end{itemize}
Then we have $m(M,g,\mathcal E)\geq 0$.
\end{thm}

We start with showing that we can deform the metric so that it is in addition conformally flat and scalar flat on the distinguished end with almost the same mass. This will be crucial to reduce the problem to non-existence problem of metrics on torus by Lohkamp's trick \cite{Lohkamp1999}.

\begin{prop}[Singular density theorem]\label{Prop: density}
Given an asymptotically flat manifold $(M,g,\mathcal E)$ with $C^0$ arbitrary ends satisfying conditions (c1) and (c2), then for any $\e>0$ we can find a complete metric $\tilde g$ on $M$ such that  $(M,\tilde g,\mathcal E)$ is an asymptotically flat manifold with $C^0$ arbitrary ends satisfying conditions (c1), (c2) and
\begin{itemize}\setlength{\itemsep}{1mm}
\item[(c3)] the metric $\tilde g$ is conformally flat around infinity of $\mathcal E$, i.e.
$$
\tilde g=\tilde u^{\frac{4}{n-2}}g_{\mathrm{euc}}\mbox{ and }R(\tilde g)\equiv 0\mbox{ around }\infty\mbox{ of }\mathcal E;
$$
\item[(c4)] we have
$$
|m(M,\tilde g,\mathcal E)-m(M,g,\mathcal E)|<\e.
$$
\end{itemize}
\end{prop}

\begin{proof}
Using $\mathcal{E}\cap S=\emptyset$, we fix neighborhoods $U_{\mathcal{E}}$ and $U_{S}$ of $\mathcal{E}$ and $S$ such that $U_{\mathcal{E}}\cap U_{S}=\emptyset$. Applying the argument of \cite[Section 3]{Lee2013} and the Sobolev embedding, there exists a family of smooth metric $g_{\delta}$ satisfying
\begin{itemize}\setlength{\itemsep}{1mm}
\item[(a1)] as $\delta\to0$, $g_{\delta}$ converges to $g$ in the $C_{\mathrm{loc}}^{\alpha}(M)$ sense for some $\alpha\in (0,1)$;
\item[(a2)] $g_{\delta}=g$ in $M-U_{S}$.
\end{itemize}
Let $\xi:\mathbb{R}\to[0,1]$ be an one-variable function such that
\begin{equation}
\xi \equiv 0 \ \text{in $(-\infty,2]$}, \ \
\xi \equiv 1 \ \text{in $[3,+\infty)$}.
\end{equation}
Set $m=m(M,g,\mathcal{E})$ and write $g$ and $g_{\delta}$ as follows:
\begin{equation}
g = \left(1+\frac{m}{r^{n-2}}\right)^{\frac{4}{n-2}}g_{\mathrm{euc}}+h, \ \
g_{\delta} = \left(1+\frac{m}{r^{n-2}}\right)^{\frac{4}{n-2}}g_{\mathrm{euc}}+h_{\delta}
\ \ \text{in $\mathcal{E}$}.
\end{equation}
For sufficiently large constant $s>1$ to be determined later, we define
\begin{equation}
\hat{g}^{s} = \left(1+\frac{m}{r^{n-2}}\right)^{\frac{4}{n-2}}g_{\mathrm{euc}}+\left(1-\xi\left(\frac{r}{s}\right)\right)h
\end{equation}
and
\begin{equation}
\hat{g}_{\delta}^{s} = \left(1+\frac{m}{r^{n-2}}\right)^{\frac{4}{n-2}}g_{\mathrm{euc}}+\left(1-\xi\left(\frac{r}{s}\right)\right)h_{\delta}.
\end{equation}
In $\mathcal{E}$, it is clear that
\begin{equation}
\hat{g}_{\delta}^{s} = \left(1+\frac{m}{r^{n-2}}\right)^{\frac{4}{n-2}}g_{\mathrm{euc}}  \ \text{in $\{r\geq 3s\}$}, \ \
\hat{g}_{\delta}^{s} = g_{\delta} \ \text{in $\{r\leq 2s\}$}.
\end{equation}
So the metric $\hat{g}_{\delta}^{s}$ can be extended to the whole manifold $M$ by defining $\hat{g}_{\delta}^{s}=g_{\delta}$ in $M-\mathcal{E}$.

In the following argument, we say a constant is uniform if it is independent of $\delta$ and $s$. By $U_{\mathcal{E}}\cap U_{S}=\emptyset$ and (a2), we obtain $g_{\delta}=g$ in $U_{\mathcal{E}}$ and hence $\hat{g}_{\delta}^{s}=\hat{g}^{s}$ in $U_{\mathcal{E}}$. The definition of $\hat{g}^{s}$ shows that $\hat{g}^{s}$ are uniformly equivalent to $g$ in $U_{\mathcal{E}}$, i.e.,
\begin{equation}
C^{-1}g \leq \hat{g}_{\delta}^{s} = \hat{g}^{s} \leq Cg \ \ \text{in $U_{\mathcal{E}}$}
\end{equation}
for some uniform constant $C$. This implies that the Sobolev inequality holds for $\hat{g}_{\delta}^{s}$ in $U_{\mathcal{E}}$ with a uniform Sobolev constant $c_{sob}$, i.e.,
\begin{equation}
c_{sob}\left(\int_{U_{\mathcal{E}}}|\vp|^{\frac{2n}{n-2}}d\mu_{\hat{g}_{\delta}^{s}}\right)^{\frac{n-2}{n}}
\leq \int_{U_{\mathcal{E}}}|\nabla\vp|^{2}d\mu_{\hat{g}_{\delta}^{s}}
\end{equation}
for any $\vp\in C^{\infty}(M)$ with compact support in $\overline{U}_{\mathcal{E}}$.

Choose another one-variable function $\eta:\mathbb{R}\to[0,1]$ such that
\begin{equation}
\eta \equiv 0 \ \text{in $(-\infty,1]\cup[4,\infty)$}, \ \
\eta \equiv 1 \ \text{in $[2,3]$}.
\end{equation}
Denote $\eta_{s}(x)=\eta(r(x)/s)$. Direct calculation using the decay assumption (cf. \cite[(17) and (18)]{Zhu2022A}) shows that there exists a constant $b_{s}$ (independent of $\delta$) such that
\begin{equation}
\left(\int_{U}\big|(R(\hat{g}_{\delta}^{s})\eta_{s}-b_{s}\eta_{s})_{-}\big|^{\frac{n}{2}}\right)^{\frac{2}{n}} \leq \frac{c_{sob}}{2}
\end{equation}
and
\begin{equation}
b_{s}(1+\mathcal{H}_{\hat{g}_{\delta}^{s}}^{n}(\{s\leq r\leq 4s\})) \leq s^{-1}.
\end{equation}
By \cite[Proposition 2.2]{Zhu2022A}, there is a positive function $u_{\delta,s}$ solving
\begin{equation}
\Delta_{\hat{g}_{\delta}^{s}}u_{\delta,s}-\frac{n-2}{4(n-1)}\big(R(\hat{g}_{\delta}^{s})\eta_{s}-b_{s}\eta_{s}\big)u_{\delta,s} = 0
\ \ \text{in $M$}.
\end{equation}
Similar calculation of \cite[(16)]{Zhu2022A} shows
\begin{equation}\label{Eq: A s 1}
\left(\int_{\{s\leq r\leq 4s\}}|R(\hat{g}_{\delta}^{s})|^{\frac{2n}{n+2}}d\mu_{\hat{g}_{\delta}^{s}}\right)^{\frac{n+2}{2n}}
\leq Cs^{-\sigma+\frac{n-2}{2}}.
\end{equation}
Applying \cite[(11)]{Zhu2022A}, we see that
\begin{equation}\label{L 2n n-2 estimate}
\begin{split}
& \left(\int_{U_{\mathcal{E}}}|u_{\delta,s}-1|^{\frac{2n}{n-2}}d\mu_{\hat{g}_{\delta}^{s}}\right)^{\frac{n-2}{2n}} \\
\leq {} & 2c_{sob}^{-1}\left(\int_{U_{\mathcal{E}}}\big|R(\hat{g}_{\delta}^{s})\eta_{s}-b_{s}\eta_{s}\big|^{\frac{2n}{n+2}}d\mu_{\hat{g}_{\delta}^{s}}\right)^{\frac{n+2}{2n}} \\
\leq {} & 2c_{sob}^{-1}\left(\int_{\{s\leq r\leq 4s\}}|R(\hat{g}_{\delta}^{s})|^{\frac{2n}{n+2}}d\mu_{\hat{g}_{\delta}^{s}}\right)^{\frac{n+2}{2n}}
+2c_{sob}^{-1} \cdot b_{s} \cdot \mathcal{H}_{\hat{g}_{\delta}^{s}}^{n}(\{s\leq r\leq 4s\}) \\[2mm]
\leq {} & 2c_{sob}^{-1}(Cs^{-\sigma+\frac{n-2}{2}}+s^{-1}) \leq Cs^{-\min\{1,\sigma-\frac{n-2}{2}\}}.
\end{split}
\end{equation}

Since $\hat{g}_{\delta}^{s}=\hat{g}^{s}$ in $U_{\mathcal{E}}$ and $\eta_{s}$ has support compact in $\{s\leq r\leq 4s\}$, the function
\begin{equation}
R(\hat{g}_{\delta}^{s})\eta_{s}-b_{s}\eta_{s}
= R(\hat{g}^{s})\eta_{s}-b_{s}\eta_{s}
\end{equation}
is independent of $\delta$. It is easy to verify that $(R(\hat{g}^{s})\eta_{s}-b_{s}\eta_{s})$ has a uniform $C_{\mathrm{loc}}^{2}(M)$ bound. By (a1), we have uniform $C_{\mathrm{loc}}^{\alpha}(M)$ bound of $g_{\delta}$. Using \eqref{L 2n n-2 estimate} and the Harnack inequality \cite[Theorem 8.20]{GilbargTrudinger2001}, we deduce that for any compact set $K\subset M$, we have 
\begin{equation}
\|u_{\delta,s}\|_{C^{0}(K)} \leq C(K),
\end{equation}
where $C(K)$ is a constant independent of $\delta$. Applying $C^{1,\alpha}$ estimate \cite[Theorem 8.32]{GilbargTrudinger2001}, we see that 
\begin{equation}
\|u_{\delta,s}\|_{C^{1,\alpha}(K)} \leq C(K).
\end{equation}
By passing to a subsequence, as $\delta\to0$, there exists $u_{s}\in C^{1,\alpha}(M)$ such that $u_{\delta,s}$ converges to $u_{s}$ in the $C_{\mathrm{loc}}^{1}(M)$ sense. Then $u_{s}$ is a weak solution of the following PDE:
\begin{equation}
\Delta_{\hat{g}^{s}}u_{s}-\frac{n-2}{4(n-1)}\big(R(\hat{g}^{s})\eta_{s}-b_{s}\eta_{s}\big)u_{s} = 0
\ \ \text{in $M$}.
\end{equation}
The standard elliptic theory shows $u_{s}$ is smooth in $U_{\mathcal{E}}$. By the Harnack inequality \cite[Theorem 8.20]{GilbargTrudinger2001}, we also know that $u_{s}$ is positive in $M$. Using \eqref{L 2n n-2 estimate} and Fatou's lemma,
\begin{equation}\label{Eq: A s 2}
\left(\int_{U_{\mathcal{E}}}|u_{s}-1|^{\frac{2n}{n-2}}d\mu_{\hat{g}^{s}}\right)^{\frac{n-2}{2n}}
\leq \liminf_{\delta\to0}\left(\int_{U_{\mathcal{E}}}|u_{\delta,s}-1|^{\frac{2n}{n-2}}d\mu_{\hat{g}_{\delta}^{s}}\right)^{\frac{n-2}{2n}}
\leq Cs^{-\min\{1,\sigma-\frac{n-2}{2}\}}.
\end{equation}
By the similar argument of \cite[Lemma 3.2]{SchoenYau1979C}, we obtain
\begin{equation}\label{Eq: asymptotic behavior}
u_{s} = 1+A_{s}r^{2-n}+\omega_{s},
\end{equation}
where
\begin{equation}\label{Eq: A_s}
A_{s} = -\frac{1}{4(n-1)|\mathbb{S}^{n-1}|}\int_{U_{\mathcal{E}}}\big(R(\hat{g}^{s})\eta_{s}-b_{s}\eta_{s}\big)u_{s}d\mu_{\hat{g}^{s}}.
\end{equation}
and
\begin{equation}
|\omega_{s}|+r|\de\omega_{s}|+r^{2}|\de^{2}\omega_{s}| \leq Cr^{1-n}.
\end{equation}
To estimate $A_{s}$, we apply the similar argument of \cite[(22)]{Zhu2022A} and obtain
\begin{equation}\label{Eq: A s 3}
\lim_{s\to+\infty}\left|\int_{\{s\leq r\leq 4s\}}R(\hat{g}^{s})\eta_{s}d\mu_{\hat{g}^{s}}\right| = 0.
\end{equation}
Define $v_{s}=u_{s}-1$ and then
\begin{equation*}
\begin{split}
& 4(n-1)\cdot|\mathbb{S}^{n-1}|\cdot|A_{s}| \leq \left|\int_{U_{\mathcal{E}}}R(\hat{g}^{s})\eta_{s}u_{s}d\mu_{\hat{g}^{s}}\right|
+\left|\int_{U_{\mathcal{E}}}b_{s}\eta_{s}u_{s}d\mu_{\hat{g}^{s}}\right| \\
\leq {} & \left(\int_{\{s\leq r\leq 4s\}}|R(\hat{g}^{s})|^{\frac{2n}{n+2}}d\mu_{\hat{g}^{s}}\right)^{\frac{n+2}{2n}}
\left(\int_{U_{\mathcal{E}}}|v_{s}|^{\frac{2n}{n-2}}d\mu_{\hat{g}^{s}}\right)^{\frac{n-2}{2n}}
+\left|\int_{\{s\leq r\leq 4s\}}R(\hat{g}^{s})\eta_{s}d\mu_{\hat{g}^{s}}\right|  \\
& +b_{s}\mathcal{H}_{\hat{g}_{\delta}^{s}}^{n}(\{s\leq r\leq 4s\})
+b_{s}\left[\mathcal{H}_{\hat{g}_{\delta}^{s}}^{n}(\{s\leq r\leq 4s\})\right]^{\frac{n+2}{2n}}
\left(\int_{U_{\mathcal{E}}}|v_{s}|^{\frac{2n}{n-2}}d\mu_{\hat{g}^{s}}\right)^{\frac{n-2}{2n}}.
\end{split}
\end{equation*}
Combining this with \eqref{Eq: A s 1}, \eqref{Eq: A s 2} and \eqref{Eq: A s 3}, we see that
\begin{equation}
\lim_{s\to+\infty}A_{s} = 0.
\end{equation}
Choosing $s$ sufficiently large, we may assume that $|A_{s}|\leq\e$ (this fixes the value of $s$).

For sufficiently small $\tau>0$ to be determined later, define
\begin{equation}
u_{s,\tau} = \frac{u_{s}+\tau}{1+\tau}, \ \ \tilde{g} = (u_{s,\tau})^{\frac{4}{n-2}}\hat{g}^{s}.
\end{equation}
We claim that $\tilde{g}$ is a complete metric satisfying (c1)-(c4). Recall that $u_{s}>0$ in $M$. The completeness of $\tilde{g}$ and (c1) follow from $u_{s,\tau}\geq\tau(1+\tau)^{-1}$ and $u_{s}\in C^{1}(M)$. In $M-S$, direct calculation shows
\begin{equation}\label{tilde g scalar}
\begin{split}
R(\tilde{g}) = {} & (u_{s,\tau})^{-\frac{n+2}{n-2}}\left(-\frac{4(n-1)}{n-2}\Delta_{\hat{g}^{s}}u_{s,\tau}+R(\hat{g}^{s})u_{s,\tau}\right) \\
= {} & (u_{s,\tau})^{-\frac{n+2}{n-2}}(1+\tau)^{-1}\Big((1-\eta_{s})R(\hat{g}^{s})u_{s}+b_{s}\eta_{s}u_{s}+R(\hat{g}^{s})\tau\Big).
\end{split}
\end{equation}
When $r\leq2s$ or $r\geq3s$, we have  $R(\hat{g}^{s})\geq0$ and hence $R(\tilde{g})\geq0$. When $2s\leq r\leq3s$, $R(\hat{g}^{s})$ might be negative. So we use the fact that $\eta_{s}=1$ to write
\begin{equation}
R(\tilde{g}) = (u_{s,\tau})^{-\frac{n+2}{n-2}}(1+\tau)^{-1}\Big(b_{s}u_{s}+R(\hat{g}^{s})\tau\Big).
\end{equation}
In $\{2s\leq r\leq3s\}$, since $u_{s}$ has a positive lower bound, then we choose sufficiently small $\tau$ such that $R(\tilde{g})\geq0$, this proves (c2).

In $\{r\geq3s\}$, $\hat{g}^{s}=\left(1+\frac{m}{r^{n-2}}\right)^{\frac{4}{n-2}}g_{\mathrm{euc}}$. Then we have $R(\hat{g}^{s})\equiv0$ and \eqref{tilde g scalar} shows $R(\tilde{g})\equiv0$. Set $\tilde{u} = u_{s,\tau}\left(1+\frac{m}{r^{n-2}}\right)$ and so $\tilde{g}=\tilde{u}^{\frac{4}{n-2}}g_{\mathrm{euc}}$ around $\infty$ of $\mathcal{E}$. We obtain (c3). On the other hand, using \eqref{Eq: asymptotic behavior},
\begin{equation}
\begin{split}
\tilde{g} = {} & \left( \frac{u_{s}+\tau}{1+\tau}\cdot\left(1+\frac{m}{r^{n-2}}\right)\right)^{\frac{4}{n-2}}g_{\mathrm{euc}} \\
= {} & \left(1+\frac{m}{r^{n-2}}+\frac{A_{s}}{1+\tau}\cdot\frac{1}{r^{n-2}}\right)^{\frac{4}{n-2}}g_{\mathrm{euc}}+\tilde{h},
\end{split}
\end{equation}
where
\begin{equation}
|\tilde{h}|+r|\de\tilde{h}|+r^{2}|\de^{2}\tilde{h}| \leq Cr^{1-n}.
\end{equation}
Then
\begin{equation}
m(M,\tilde{g},\mathcal{E}) = m+\frac{A_{s}}{1+\tau}
\end{equation}
and so
\begin{equation}
|m(M,\tilde{g},\mathcal{E})-m(M,g,\mathcal{E})| \leq |A_{s}| \leq \e,
\end{equation}
which shows (c4).
\end{proof}

The next proposition is a modification of Lohkamp's compactification \cite{Lohkamp1999}, which reduces the positive mass theorem to the non-existence problem on torus.
\begin{prop}[Lohkamp compactification]\label{Prop: compactification}
If $(M,\tilde g,\mathcal E)$ is an asymptotically flat manifold with $C^0$ arbitrary ends satisfying conditions (c1), (c2) and (c3), and its ADM mass $m(M,g,\mathcal E)$ is negative, then we can construct a complete Riemannian manifold $(\hat M,\hat g)$ such that there are three bounded open subsets $U_0\subset U_1\subset U_2$ such that
\begin{enumerate}\setlength{\itemsep}{1mm}
\item the metric $\hat{g}$ is in $W^{1,p}_{\mathrm{loc}}(\hat{M})\cap C^\infty(\hat{M}-S)$, where $S$ is a closed subset of $\hat{M}$ disjoint from $\bar U_2$ with $\mathcal{H}_{\mathrm{loc}}^{n-\frac{p}{p-1}}(S)<\infty$ if $n<p<\infty$ or $\mathcal{H}^{n-1}(S)=0$ if $p=\infty$;
\item $R(\hat g)\geq 0$ in $M-S$;
\item $U_0$ is diffeomorphic to $\mathbb{T}^n-B$, where $B$ is a ball in $\mathbb{T}^{n}$;
\item $\hat D_0:=\dist_{\hat g}(\partial U_2,U_1)>0$;
\item $R(\hat g)>0$ in $\mathcal A:=\bar U_2-U_1$.
\end{enumerate}
\end{prop}
\begin{proof}
We apply the idea of Lohkamp compactification from \cite[Proposition 6.1]{Lohkamp1999}. From condition (c3) and negative ADM mass, we see that the conformal factor $\tilde u$ is a harmonic function in an exterior region of the Euclidean space such that
\begin{equation}
\tilde u=1+Ar^{2-n}+O(r^{1-n})\mbox{ with }A<0.
\end{equation}
So we can take $s_1$ large enough such that $\tilde u$ is harmonic, $\tilde u<1$ and $\nabla \tilde u\neq 0$ in $\{r\geq s_1\}$. Denote
\begin{equation}
\e=\frac{1}{4}\left(1-\sup_{r(x)=s_1}\tilde u(x)\right).
\end{equation}
It is clear that $\tilde u>1-\e$ in $\{r\geq s_2\}$ for sufficiently large $s_2>s_1$. Take a cutoff function $\zeta:[0,+\infty)\to [0,1-2\e]$ such that $\zeta(t)=t$ when $t\leq 1-3\e$ and $\zeta(t)=1-2\e$ when $t\geq 1-\e$. Moreover, we can also require $\zeta'\geq 0$ and $\zeta''\leq 0$ in $[0,+\infty)$ as well as $\zeta''<0$ in $(1-3\e,1-\e)$. Such a cutoff function is illustrated as in Figure \ref{Fig1}.

\begin{figure}[htbp]
\centering
\includegraphics[width=7cm]{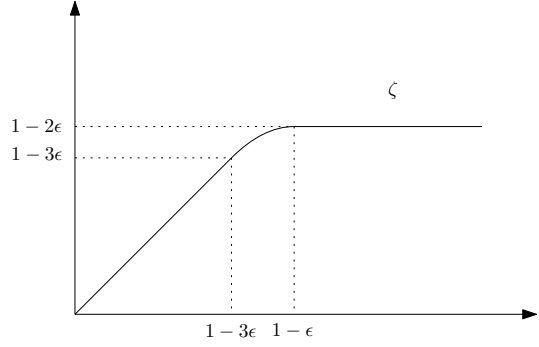}
\caption{The cutoff function $\zeta$}
\label{Fig1}
\end{figure}

Let
\begin{equation}
\begin{split}
\tilde g_{\mathrm{bend}}=\left\{
\begin{array}{cc}
(\zeta\circ \tilde u)^{\frac{4}{n-2}}g_{\mathrm{euc}},&r\geq s_1;\\[2mm]
\tilde g,& \mbox{otherwise}.
\end{array}\right.
\end{split}
\end{equation}
Clearly $\tilde g_{\mathrm{bend}}$ has the same regularity as $\tilde g$ and $\tilde g$ is just the Euclidean metric around the infinity of $\mathcal E$. Since $\tilde u$ is harmonic in $\{r\geq s_1\}$, then
\begin{equation}
R(\tilde g_{\mathrm{bend}})=-\frac{4(n-1)}{n-2}\cdot\tilde u^{-\frac{n+2}{n-2}}\cdot|\nabla\tilde u|^2\cdot(\zeta''\circ \tilde u) \ \mbox{ when }r\geq s_1.
\end{equation}
As a consequence we have $R(g_{\mathrm{bend}})>0$ in $\{r\geq s_1\}\cap\{1-3\e<\tilde u<1-\e\}$. Let us take
\begin{equation}
V_2=\{r>s_1\}\cap\left\{\tilde u>1-\frac{5}{2}\e\right\}
\end{equation}
and
\begin{equation}
V_1=\{r>s_1\}\cap \left\{\tilde u>1-\frac{3}{2}\e\right\}.
\end{equation}
Then $R(\tilde g_{\mathrm{bend}})>0$ in $\bar V_2-V_1$. Take $s_3>s_1$ large enough such that $V_0:=\{r>s_3\}$ is contained in $\{\tilde u>1-\e\}\subset V_1$. After gluing opposite faces of the cube
\begin{equation}
C_{2s_3}:=\{x:-2s_3\leq x_i\leq 2s_3,\,i=1,2,\ldots,n\}
\end{equation}
we can close the manifold $M$ in the end $\mathcal E$ and let us denote the new manifold by $\hat M$. Let $\tilde M_{2s_3}:=\tilde M-\{x\in \mathcal E: |x_i|>2s_3\mbox{ for some }i\}$ and we denote the gluing map by $\Phi: \tilde M_{2s_3}\to \hat M$. Define $\hat g=\Phi_*\tilde g$ and
\begin{equation}
U_i=\Phi(V_i\cap C_{2s_3})\mbox{ for }i=0,1,2.
\end{equation}
Now it is easy to verify that $(\hat M,\hat g)$ associated with $U_0\subset U_1\subset U_2$ satisfies all our requirements.
\end{proof}

The next proposition shows that we can perturb the rough metric slightly so that it becomes a smooth metric while the shielding estimate is almost preserved.

\begin{prop}[Local smoothing keeping shielding condition]\label{Prop: smoothing}
Let $(N,\hat g)$ be a compact Riemannian manifold with boundary associated with a finite open exhaustion
\begin{equation}
\mathbb{T}^n-B\approx U_0\subset U_1\subset U_2\subset N
\end{equation}
where $B$ is a ball in $\mathbb{T}^{n}$. Assume that
\begin{enumerate}\setlength{\itemsep}{1mm}
\item the metric $\hat{g}$ is in $W^{1,p}_{\mathrm{loc}}(N)\cap C^\infty(N-S)$, where $S$ is a closed subset of $N$ disjoint from $\bar U_2$ with $\mathcal{H}_{\mathrm{loc}}^{n-\frac{p}{p-1}}(S)<\infty$ if $n<p<\infty$ or $\mathcal{H}^{n-1}(S)=0$ if $p=\infty$;
\item $R(\hat g)\geq 0$ in $N-S$;
\item the shielding condition holds:
\begin{equation}\label{Eq: shielding condition}
\inf_{\bar{\mathcal A}}R(\hat g)\cdot \hat D_0 \cdot \hat D_1\geq \alpha_0>0,
\end{equation}
where $\mathcal A=U_2-\bar U_1$, $\hat D_0=\dist_{\hat g}(\partial U_2,U_1)$ and $\hat D_1=\dist_{\hat g}(\partial N,U_2)$.
\end{enumerate}
Then for any $\e>0$, we can find a smooth metric $\bar g$ such that $(N,\bar g)$ has $\mathbb{T}^1$-stablized scalar curvature lower bound $\sigma\geq 0$ such that
\begin{equation}
\inf_{\bar{\mathcal A}}\sigma\cdot \bar D_0 \cdot \bar D_1\geq (1-\e)\alpha_0,
\end{equation}
where
\begin{equation}
\bar D_0=\dist_{\bar g}(\partial U_2,U_1)\mbox{ and }\bar D_1=\dist_{\bar g}(\partial N,U_2).
\end{equation}
\end{prop}
\begin{proof}
Fix an open neighborhood $U_S$ of $S$ whose closure $\bar U_S$ is disjoint from $\bar U_2$. For any small $\e>0$, it follows from \cite[Lemma 2.6, Remark 2.3 and Lemma 2.7]{JiangShengZhang2020} and the classical mollification method that we can construct a smooth approximation metric $\bar g$ such that
\begin{enumerate}\setlength{\itemsep}{1mm}
\item[(i)] $\bar g$ coincides with $\hat g$ outside $U_S$;
\item[(ii)] $ (1-\e/2)\hat g\leq \bar g\leq (1+\e/2)\hat g$ as quadratic forms;
\item[(iii)] for any smooth function $\phi$ vanishing on $\partial N$ we have
\begin{equation}
\int_{N}R(\bar g)\phi^2 d\mu_{\bar g}\geq -\e \int_M|\nabla_{\bar g} \phi|^2 d\mu_{\bar g}.
\end{equation}
\end{enumerate}
Take a smooth function $\eta: N\to [0,1]$ such that $\eta\equiv 1$ in $U_{S}$ and $\eta\equiv 0$ in $U_2$. Let $u$ be the first eigenfunction such that
\begin{equation}
-\Delta_{\bar g}u+\frac{\eta^2}{2}R(\bar g)u=\lambda u\mbox{ in }N\mbox{ and }u=0\mbox{ on }\partial N.
\end{equation}
Through integration by parts we have
\begin{equation}
\begin{split}
\lambda\int_Nu^2 d\mu_{\bar g}&=\int_{N}|\nabla_{\bar g} u|^2 d\mu_{\bar g}+\frac{1}{2}\int_NR(\bar g)(\eta u)^2 d\mu_{\bar g}.
\end{split}
\end{equation}
Notice that the function $\eta u$ vanishes on $\partial N$ and so
\begin{equation}
\begin{split}
\int_NR(\bar g)(\eta u)^2 d\mu_{\bar g}&\geq -\e\int_N|\nabla_{\bar g}(\eta u)|^2 d\mu_{\bar g}\\
&\geq -2\e\left(\int_N|\nabla_{\bar g}u|^2 d\mu_{\bar g}+\int_Nu^2|\nabla_{\bar g}\eta|^2 d\mu_{\bar g}\right).
\end{split}
\end{equation}
From the Sobolev inequality, we see
\begin{equation}
\begin{split}
\int_Nu^2|\nabla_{\bar g}\eta|^2 d\mu_{\bar g}&\leq \left(\int_N|\nabla_{\bar g}\eta|^n d\mu_{\bar g}\right)^{\frac{2}{n}} \left(\int_N u^\frac{2n}{n-2} d\mu_{\bar g}\right)^{\frac{n-2}{n}} \\
&\leq C_0\int_N|\nabla_{\bar g}u|^2 d\mu_{\bar g},
\end{split}
\end{equation}
where $C_0$ is a universal constant depending only on $(N,\hat g)$ and $\eta$. Now we arrive at
\begin{equation}
\lambda\int_Nu^2\mathrm d\mu_{\bar g}\geq (1-\e-C_0\e)\int_N|\nabla_{\bar g}u|^2 d\mu_{\bar g}.
\end{equation}
By taking $\e$ small enough, we see that the first Dirichlet eigenvalue $\lambda$ is positive. Notice that
\begin{equation}
R(\bar g+u^2 dt^2)=R(\bar g)-2u^{-1}\Delta_{\bar g}u=(1-\eta^2)R(\bar g)+2\lambda\geq 0.
\end{equation}
Denote $\sigma=(1-\eta^2)R(\bar g)+2\lambda$. Then we conclude that $(N,\bar g)$ has $\mathbb{T}^1$-stablized scalar curvature $\sigma$ and $\sigma$ satisfies
\begin{equation}
\inf_{\bar{\mathcal A}}\sigma\cdot\bar D_0 \cdot \bar D_1\geq \inf_{\bar{\mathcal A}} R(\hat g)\cdot (1-\e/2)\hat D_0\cdot(1-\e/2)\hat D_1 \geq (1-\e)\alpha_0.
\end{equation}
This completes the proof.
\end{proof}

Now, we are ready to prove the main result of positive mass theorem.
\begin{proof}[Proof of Theorem \ref{Thm: nonnegative mass}]
We argue by contradiction. Suppose that the ADM mass $m(M,g,\mathcal E)$ is negative. After choosing $\e$ small enough in Proposition \ref{Prop: density}, we can find a complete metric $\tilde g$ on $M$ such that $(M,\tilde g,\mathcal E)$ is an asymptotically flat manifold with $C^0$ arbitrary ends satisfying (c1), (c2) and (c3), whose ADM mass $m(M,\tilde g,\mathcal E)$ is still negative. Now we can apply Proposition \ref{Prop: compactification} and obtain a complete Riemannian manifold $(\hat M,\hat g)$ such that there are three bounded open subsets $U_0\subset U_1\subset U_2$ such that
\begin{enumerate}\setlength{\itemsep}{1mm}
\item the metric $\hat{g}$ is in $W^{1,p}_{\mathrm{loc}}(\hat{M})\cap C^\infty(\hat{M}-S)$, where $S$ is a closed subset of $\hat{M}$ disjoint from $\bar U_2$ with $\mathcal{H}_{\mathrm{loc}}^{n-\frac{p}{p-1}}(S)<\infty$ if $n<p<\infty$ or $\mathcal{H}^{n-1}(S)=0$ if $p=\infty$;
\item $R(\hat g)\geq 0$ in $M-S$;
\item $U_0$ is diffeomorphic to $\mathbb{T}^n-B$, where $B$ is a ball in $\mathbb{T}^{n}$;
\item $\hat D_0:=\dist_{\hat g}(\partial U_2,U_1)>0$;
\item $R(\hat g)>0$ in $\mathcal A:=\bar U_2-U_1$.
\end{enumerate}
Take $N$ to be the $\hat D_1$-neighborhood of $U_2$ such that
\begin{equation}
\alpha_0:=\inf_{\bar{\mathcal A}}R(\hat g)\cdot \hat D_0 \cdot \hat D_1>5.
\end{equation}
It follows from Proposition \ref{Prop: smoothing} that we can find a smooth metric $\bar g$ such that $(N,\bar g)$ has $\mathbb{T}^1$-stablized scalar curvature lower bound $\sigma\geq 0$ such that
\begin{equation}
\inf_{\bar{\mathcal A}}\sigma\cdot \bar D_0 \cdot \bar D_1>4,
\end{equation}
where
\begin{equation}
\bar D_0=\dist_{\bar g}(\partial U_2,U_1)\mbox{ and }\bar D_1=\dist_{\bar g}(\partial N,U_2).
\end{equation}
This leads to a contradiction to Theorem \ref{Thm: Geroch conjecture}.
\end{proof}

\section{Rigidity under vanishing mass}

In this section, we discuss the case when $m(M,g,\mathcal{E})=0$. When there is no singular set, it was shown by the third named author \cite{Zhu2022A} and Lee-Lesourd-Unger \cite{LeeLesourdUnger2022}  independently that $M$ is isometric to the standard Euclidean space. On the other hand, if the singular set is bounded in an asymptotically flat manifold, the rigidity was also obtained by Jiang-Sheng-Zhang \cite{JiangShengZhang2020}. In this work, our main objective is to force $M$ to have a single ended when the mass of an end vanishes and hence the singular set $S$ is bounded.

\subsection{Ricci flatness outside  $S$}
In this subsection, we prove the Ricci flatness away from the singularity. This is a straight forward adaption of the method developed in \cite[Theorem 1.2]{Zhu2022A} with modification from \cite{JiangShengZhang2020}.

\begin{prop}\label{Prop:ricci-flat}
Under the assumption of Theorem~\ref{intro-main}, if $m(M,g,\mathcal{E})=0$, then $\Ric(g)\equiv0$ in $M-S$.
\end{prop}
\begin{proof}
The proof is almost identical to that of \cite[Theorem 1.2]{Zhu2022A}, see also \cite{LeeLesourdUnger2022} and \cite{LiMantoulidis2019}. We only give a sketch of it for reader's convenience.

We first show that the scalar curvature vanishes outside $S$. Suppose on the contrary, there is $z\notin S$ such that $R(g)(z)>0$. We choose a sufficiently small neighborhood $U$ around $z$ such that $R(g)>0$ in $U$.  Take a nonnegative cutoff function $\eta$ with compact support in $U$ such that $\eta=1$ around $z$. It follows from the argument in Proposition~\ref{Prop: density} (see also \cite[Proposition 2.2]{Zhu2022A}) that there is a positive function $u$ solving
\begin{equation}
\Delta_{g}u-\frac{n-2}{4(n-1)}\eta R(g)u=0
\end{equation}
with the expansion $u=1+Ar^{2-n}+O(r^{1-n})$ where $A<0$ from the formula \eqref{Eq: A_s}. Then the metric $\bar g=\left(\frac{u+1}{2} \right)^\frac{4}{n-2}g$ is a complete metric with $R(\bar g)\geq 0$ and negative ADM mass. We remark that due to the cutoff, the above equation is the usual harmonic equation around the singularity and hence only the operator $\Delta_g$ is (mildly) singular. To obtain contradiction, it suffices to show that $\bar g$ still satisfies the assumption of Theorem~\ref{intro-main}.

By the Harnack inequality \cite[Theorem 8.20]{GilbargTrudinger2001} and $C^{1,\alpha}$ estimate \cite[Theorem 8.32]{GilbargTrudinger2001}, it is clear that $u$ is bounded and $u\in C^{1,\a}_{\mathrm{loc}}(M)\cap C^\infty_{\mathrm{loc}}(M-S)$ for some $\a\in(0,1)$ and hence $\bar g\in W^{1,p}_{\mathrm{loc}}(M)\cap C^\infty_{\mathrm{loc}}(M-S)$.
Thanks to the local boundedness of $u$, we still have $\mathcal{H}_{\bar{g},\mathrm{loc}}^{n-\frac{p}{p-1}}(S)<\infty$ if $n<p<\infty$ or $\mathcal{H}_{\bar{g}}^{n-1}(S)=0$ if $p=\infty$. Then Theorem \ref{intro-main} implies that the mass is nonnegative which is impossible. This proves the scalar flatness away from $S$.

To prove the Ricci flatness outside $S$,  we fix an arbitrary point $p\notin S$ and consider the perturbation $\tilde g=g-t \eta \Ric(g)$ where $\eta$ is a smooth cutoff around $p$ and away from $S$ which still satisfies the same assumptions as $g$.  If $\Ric(g)$ is non-zero at $p$, then for $t$ small we can use the same argument of \cite[Theorem 1.2]{Zhu2022A} with modification as in the argument of Proposition~\ref{Prop: density} to find a positive function $\tilde u$ solving
\begin{equation}
\Delta_{\tilde g}\tilde u-\frac{n-2}{4(n-1)}\big(R(\tilde{g})-\tilde b\tilde\eta\big)\tilde u= 0
\ \ \text{in $M$}.
\end{equation}
with expansion $\tilde u=1+\tilde Ar^{2-n}+O(r^{1-n})$ for some constant $\tilde A<0$, where $\bar\eta$ is a cutoff function supported in $M-S$ such that $\bar\eta \equiv 1$ in the support of $\eta$ and $\tilde b$ is a sufficiently small positive constant. As before, we take $\bar g=\left(\frac{\tilde u+\tau}{1+\tau} \right)^\frac{4}{n-2}\tilde g$ for a sufficiently small positive constant $\tau$ and verify as above that $\bar g$ is a complete metric on $M$ with $R(\bar g)\geq 0$ and negative ADM mass, which leads to a contradiction.
\end{proof}

\subsection{Isometry when $p$ is finte}
In this subsection, we discuss the case of $3\leq n<p<\infty$. Our goal is to show that Ricci flatness outside singularity implies a rigidity on the topological type. 
 Precisely, we will prove the isometry part of Theorem \ref{intro-rigidity} when $p$ is finite. The existence of $C^{1,\alpha}_{\mathrm{loc}}$ diffeomorphism will be proved in the next subsection.

We start with the following regularity result which can be proved by a similar argument as in \cite[Theorem 6.1]{ShiTamPJM}.
\begin{lma}\label{Lem: L p regularity}
Let $B_1$ be the unit ball in the Euclidean space $\mathbb{R}^{n}$ where $n\geq3$. Suppose that $g$ is a metric in $B_1$ and there is a closed subset $S$ with
$\mathcal{H}^{n-\frac{p}{p-1}}(S)<\infty$ such that $g\in W^{1,p}(B_1)\cap C^\infty(B_1-S)$ for some $n<p<\infty$. If $F\in W^{1,p}(B_1)\cap C^\infty(B_1-S)$ is a function satisfying
\begin{equation}
g^{ij}\partial_{i}\partial_{j}F = Q \ \ \text{in $ B_{1}-S$}
\end{equation}
for some $Q\in L^{q}(B_{1})$, where $q$ satisfies
\begin{equation}\label{Eq: condition for q}
\frac{nq}{n-q} > 2.
\end{equation}
Then we have $F\in W_{\mathrm{loc}}^{2,q}(B_1)$.
\end{lma}
\begin{proof}
By the Sobolev embedding, $g$ is continuous in $B_1$. We may assume $g\in C(\bar B_1)$ and $\lambda g_{\mathrm{euc}}\leq g\leq \Lambda g_{\mathrm{euc}}$ for some positive constants $\lambda<\Lambda$, otherwise we can work in a smaller ball. We split the argument into two cases:

\medskip
\noindent
{\bf Case 1.} $q<n$.
\medskip

Fix a smooth cutoff function $\eta:B_1\to[0,1]$ supported in $B_1$ such that $\eta\equiv 1$ in $B_{1/2}$, then we can compute
\begin{equation}\label{Eq: equation F}
g^{ij}\partial_i\partial_j(\eta F)=\tilde Q \ \ \text{in $ B_{1}-S$},
\end{equation}
where
\begin{equation}
\tilde Q=Fg^{ij}\partial_i\partial_j\eta +2g^{ij}(\partial_i F)(\partial_j\eta)+\eta Q\in L^q(B_1).
\end{equation}
From \cite[Theorem 9.15]{GilbargTrudinger2001}, we can find a function $G\in W^{2,q}(B_1)\cap W^{1,q}_0(B_1)$ such that
\begin{equation}
g^{ij}\partial_i\partial_j G=\tilde Q\ \ \text{in}\ B_1.
\end{equation}
From the Sobolev embedding, we see that $G\in W^{1,\frac{nq}{n-q}}_0(B_1)$. Using condition \eqref{Eq: condition for q} and $p>n\geq3$,
\begin{equation}
\frac{nq}{n-q} > 2 > \frac{n}{n-1} > \frac{p}{p-1}.
\end{equation}
Therefore for any $\varphi\in C^\infty_0(B_1)$, we have
\begin{equation}\label{Eq: weak eq}
-\int_{B_1}(\partial_ig^{ij})(\partial_jG)\varphi+g^{ij}(\partial_jG)(\partial_i\varphi)\, dx=\int_{B_1}\tilde Q\varphi\,dx.
\end{equation}

Now we claim that \eqref{Eq: weak eq} also holds for $\eta F$.
It follows from \cite[Lemma A.1]{JiangShengZhang2020} that we can take a family of cutoff functions $\zeta_\e:B_1\to [0,1]$ such that $\zeta_\e$ vanishes in a neighborhood of $S$ and $\zeta_\e\equiv 1$ outside $B_\e(S)$. Moreover, we have
\begin{equation}\label{Eq: cutoff function}
\lim_{\e\to 0}\int_{B_1}|\de\zeta_\e|^\frac{p}{p-1} \, dx =0  \ \text{and} \ \lim_{\e\to 0}\mathcal H^n(\{\zeta_\e\neq 1\})=0.
\end{equation}
We multiply both sides of \eqref{Eq: equation F} by $\zeta_\e\varphi$ and it follows from integration by parts that
\begin{equation*}
-\int_{B_1}\partial_ig^{ij}\cdot\partial_j(\eta F)\cdot\zeta_\e\cdot\varphi+g^{ij}\cdot\partial_j(\eta F)
\cdot(\partial_i\varphi \cdot\zeta_\e+\varphi\cdot\partial_i\zeta_\e)\, dx
=\int_{B_1}\tilde Q\zeta_\e\varphi\, dx.
\end{equation*}
It is easy to check
\begin{equation*}
\left|\int_{B_1}\partial_ig^{ij}\cdot\partial_j(\eta F)\cdot(\zeta_\e-1)\cdot\varphi\, dx\right|\leq C\|\partial g\|_{L^p(B_1)}\cdot\|F\|_{W^{1,p}(B_1)}\cdot\|\zeta_\e-1\|_{L^{\frac{p}{p-2}}(B_1)},
\end{equation*}
\begin{equation}
\left|\int_{B_1}g^{ij}\cdot\partial_j(\eta F)\cdot\partial_i\varphi\cdot(\zeta_\e-1)\, dx\right|
\leq C\|F\|_{W^{1,p}(B_1)}\cdot\|\zeta_\e-1\|_{L^{\frac{p}{p-1}}(B_1)}
\end{equation}
and
\begin{equation}
\left|\int_{B_1}g^{ij}\cdot\partial_j(\eta F)\cdot\varphi\cdot\partial_i\zeta_\e\, dx\right|
\leq C\|F\|_{W^{1,p}(B_1)}\cdot\|\partial \zeta_\e\|_{L^{\frac{p}{p-1}}(B_1)}.
\end{equation}
Using \eqref{Eq: cutoff function}, the claim is now verified by letting $\e\to 0$.

Denote $H=\eta F-G$. Then \eqref{Eq: weak eq} and the claim show that for any $\varphi\in C^\infty_0(B_1)$,
\begin{equation}\label{Eq: equation H}
\int_{B_1}(\partial_ig^{ij})(\partial_jH)\varphi+g^{ij}(\partial_jH)(\partial_i\varphi)\, dx=0.
\end{equation}
Since $F\in W^{1,p}(B_{1})$ and $G\in W^{1,\frac{nq}{n-q}}_0(B_1)$, we see that $H\in W^{1,q^*}_0(B_1)$, where $q^{*}=\min(p,\frac{nq}{n-q})$. Using condition \eqref{Eq: condition for q}, $g\in W^{1,p}(B_1)$ for $p>n$ and approximation it is standard to see \eqref{Eq: equation H} also holds for any $\varphi\in W^{1,q^*}_0(B_1)$. Next we repeat the argument as in \cite[Theorem 8.1]{GilbargTrudinger2001} to show $H\equiv 0$, which yields that $F=G$ in $B_{1/2}$ and so $F\in W^{2,q}(B_{1/2})$. Then $F\in W^{2,q}_{\mathrm{loc}}(B_{1})$ follows from a covering argument.

We claim that $H\leq 0$. Otherwise for any $0\leq k<\sup_{B_1}H$, we take $\varphi_k=(H-k)^+$ to be the nonnegative part of $(H-k)$. It is clear that $\varphi_k\in W^{1,q^*}_0(B_1)$ and \eqref{Eq: equation H} implies
\begin{equation}
\begin{split}
\int_{B_1}|\de\varphi_k|^2\, dx
\leq {} & C\int_{B_1}|\partial_ig^{ij}|\cdot|\partial_jH|\cdot\varphi_{k}\, dx\\[1mm]
\leq {} & C\|\partial g\|_{L^p(B_1)}\cdot\|(\partial H)\cdot\varphi_k\|_{L^{\frac{p}{p-1}}(B_1)}\\[1.5mm]
\leq {} & C\|\partial g\|_{L^p(B_1)}\cdot\|\de\varphi_k\|_{L^2(B_1)}\cdot\|\varphi_k\|_{L^{\frac{2p}{p-2}}(\Gamma_{k})},
\end{split}
\end{equation}
where $\Gamma_{k}$ is the support of $\de\varphi_k$. Thanks to $g\in W^{1,p}(B_1)$ for $p>n$, then we have
\begin{equation}
\|\de\varphi_k\|_{L^2(B_1)} \leq C\|\varphi_k\|_{L^{\frac{2p}{p-2}}(\Gamma_{k})}.
\end{equation}
From condition \eqref{Eq: condition for q} and $\vp_{k}\in W^{1,q^*}_0(B_1)$ where $q^{*}=\min(p,\frac{nq}{n-q})$, we see that $\varphi_k\in W^{1,2}_0(B_1)$. Combining the above with the Sobolev inequality,
\begin{equation*}
\|\varphi_k\|_{L^{\frac{2n}{n-2}}(B_1)} \leq C\|\de\varphi_k\|_{L^2(B_1)}
\leq C\|\varphi_k\|_{L^{\frac{2p}{p-2}}(\Gamma_{k})}
\leq C\left(\mathcal H^n(\Gamma_{k})\right)^{\frac{\alpha-1}{\alpha}\cdot\frac{p-2}{2p}}\|\varphi_k\|_{L^{\frac{2n}{n-2}}(B_1)},
\end{equation*}
where $\alpha=\frac{n(p-2)}{p(n-2)}$. Hence, we have $\mathcal H^n(\Gamma_{k})\geq c_0$ for some positive constant $c_0$ independent of $k$. In particular, from $H\geq k$ in $\Gamma_{k}$, we see
\begin{equation}
\mathcal H^n(\{H\geq k\})\geq c_0.
\end{equation}
Since $H$ is $L^{q^*}$-integrable, we know that $\sup_{B_{1}}H$ is finite. After letting $k$ tend to $\sup_{B_{1}}H$, we can find a subset $\tilde \Gamma$ with positive measure where we have $H=\sup_{B_{1}}H$ and $\de H\neq 0$. This is impossible and so the claim is proved. By the same argument, we obtain $H\geq0$ and then $H\equiv0$, as desired.

\medskip
\noindent
{\bf Case 2.} $q\geq n$.
\medskip

For all $s<n$, it is clear that $Q\in L^{q}(B_{1})\in L^{s}(B_{1})$. By Case 1, we see that $F\in W^{2,s}_{\mathrm{loc}}(B_{1})$, which implies $F\in W^{1,p'}_{\mathrm{loc}}(B_{1})$ for all $p'$. We assume without loss of generality that $F\in W^{1,p'}(B_{1})$, otherwise we can work in a smaller ball. By the same construction of Case 1, we obtain $\tilde{Q}\in L^{q}(B_{1})$. Since $q\geq n$, then $G\in W^{2,q}(B_{1})\cap W_{0}^{1,q}(B_{1})\subset W_{0}^{1,q'}(B_{1})$ for all $q'$. It follows that $H=\eta F-G\in W_{0}^{1,q^{*}}(B_{1})$ for all $q^*$. The same argument shows $H\equiv0$ and then $F\in W^{2,q}_{\mathrm{loc}}(B_{1})$.
\end{proof}

Now we are in a position to prove the isometry part of Theorem~\ref{intro-rigidity} when $p$ is finite.

\begin{thm}\label{thm:one-end-1}
Under the assumption of Theorem~\ref{intro-rigidity}, if $n<p<\infty$, then $S$ is bounded and $(M,g)$ is isometric to the standard Euclidean space as a metric space. Moreover, $g$ is flat outside $S$.
\end{thm}
\begin{proof}
For any $z\in S$, we fix a coordinate system $(U_{z},\{x^{i}\}_{i=1}^{n})$ near $z$. By \cite[Theorem 1.1]{JulinLiimatainenSalo2017}, there exists another coordinate system $(V_{z},\{y^{a}\}_{a=1}^{n})$ near $z$ such that $V_{z}\subset U_{z}$ and each $y^{a}$ is a harmonic function with $C^{1}$ regularity. Thanks to Lemma \ref{Lem: L p regularity}, the $C^{1}$ regularity can be lifted to $W^{2,p}$ regularity. Write
\begin{equation}
g_{ij} = g\left(\frac{\de}{\de x^{i}},\frac{\de}{\de x^{j}}\right), \ \
h_{ab} = g\left(\frac{\de}{\de y^{a}},\frac{\de}{\de y^{b}}\right).
\end{equation}
Then
\begin{equation}
h_{ab} = \frac{\de x^{i}}{\de y^{a}}\,\frac{\de x^{j}}{\de y^{b}}\,g_{ij}, \ \ \
\frac{\de h_{ab}}{\de y^{c}} = 2\,\frac{\de^{2}x^{i}}{\de y^{a}\de y^{c}}\,\frac{\de x^{j}}{\de y^{b}}\, g_{ij}
+\frac{\de x^{i}}{\de y^{a}}\,\frac{\de x^{j}}{\de y^{b}}\,\frac{\de x^{k}}{\de y^{c}}\,\frac{\de g_{ij}}{\de x^{k}}.
\end{equation}
Combining this with the assumption $g\in W^{1,p}(U_{z})$ and $W^{2,p}$ regularity of $y^{a}$, we see that $h\in W^{1,p}(V_{z})$.

On the other hand, Proposition \ref{Prop:ricci-flat} shows $\Ric\equiv0$ in $M-S$. Since $y^{a}$ is harmonic, we have
\begin{equation}
h^{cd}\frac{\de^{2} h_{ab}}{\de y^{c}\de y^{d}} = Q(h,\de_{y}h) \ \ \text{in $V_{z}-S$},
\end{equation}
where
\begin{equation}
Q(h,\de_{y}h) = \de_{y}h^{-1}*\de_{y}h+h^{-1}*h^{-1}*\de_{y}h*\de_{y}h.
\end{equation}

We claim that for $p\leq p'<2n$, if $h\in W_{\mathrm{loc}}^{1,p'}(V_{z})$, then $h\in W_{\mathrm{loc}}^{1,\theta p'}(V_{z})$ where $\theta:=\frac{2n}{3n-p}>1$. It is clear that $Q\in L_{\mathrm{loc}}^{\frac{p'}{2}}(V_{z})$ and Lemma \ref{Lem: L p regularity} shows $h\in W_{\mathrm{loc}}^{2,\frac{p'}{2}}(V_{z})$. By the Sobolev embedding, $h\in W_{\mathrm{loc}}^{1,\frac{np'}{2n-p'}}(V_{z})$. It is clear that
\begin{equation}
\frac{np'}{2n-p'} \geq \frac{np'}{2n-\frac{p+n}{2}} = \theta p'.
\end{equation}
and then $h\in W_{\mathrm{loc}}^{1,\theta p'}(V_{z})$.

Combining the above claim and assumption $h\in W^{1,p}(V_{z})$, we obtain $h\in W_{\mathrm{loc}}^{1,\theta p'}(V_{z})$ for all $p\leq p'<2n$ and then $Q\in L_{\mathrm{loc}}^{\frac{\theta p'}{2}}(V_{z})$. By Lemma \ref{Lem: L p regularity}, $h\in W_{\mathrm{loc}}^{2,\frac{\theta p'}{2}}(V_{z})$. Choosing $p'$ such that $p'>\frac{2n}{\theta}$, the Sobolev embedding shows $h\in C_{\mathrm{loc}}^{1}(V_{z})$ and so $Q\in L_{\mathrm{loc}}^{\infty}(V_{z})$. Using Lemma \ref{Lem: L p regularity} again, we see that $h\in W_{\mathrm{loc}}^{2,q}(V_{z})$ for all $q$. Combining this \cite[Theorem 9.19]{GilbargTrudinger2001}, we obtain $h\in W_{\mathrm{loc}}^{3,q}(V_{z})$. Repeating such argument, we obtain $h$ is actually smooth.

Covering $S$ by the above harmonic coordinate system $V_{z}$ and using \cite[Theorem 2.1]{Taylor2006}, we obtain a new smooth structure. Under this new smooth structure, the metric $g$ is smooth and $\Ric\equiv0$ in $M$. Fix a point $z_{0}\in M$. By the volume comparison theorem,
\begin{equation}
\Vol(B_{R}(z_{0}))\leq |B_{\mathrm{euc}}(R)| \ \ \text{and} \ \
\frac{\Vol(B_{R}(z_{0}))}{|B_{\mathrm{euc}}(R)|} \ \ \text{is non-increasing},
\end{equation}
where $|B_{\mathrm{euc}}(R)|$ denotes the volume of ball with radius $R$ in $\mathbb{R}^{n}$.
Since $\mathcal{E}$ is asymptotically flat, then
\begin{equation}
\lim_{R\to+\infty}\frac{\Vol(B_{R}(z_{0}))}{|B_{\mathrm{euc}}(R)|} \geq 1.
\end{equation}
Combining the above, we obtain $\Vol(B_{R}(z_{0}))=|B_{\mathrm{euc}}(R)|$ for any $R>0$, which implies that $M$ with the new smooth structure is isometric to the Euclidean space. It follows that $M$ has only one end $\mathcal{E}$. Then $S\cap\mathcal{E}=\emptyset$ shows $S$ is bounded. Since the smooth structure outside $S$ coincides with the original one, we conclude that $g$ is flat outside $S$. Moreover, the distance isometry follows from \cite[Theorem 1.1]{JiangShengZhang2020}.
\end{proof}

\subsection{Proof of Theorem \ref{intro-rigidity}}
In the previous subsection, we prove the isometry part of Theorem \ref{intro-rigidity} when $p$ is finite. For the existence of $C^{1,\alpha}_{\mathrm{loc}}$ diffeomorphism,  the proof of Theorem~\ref{thm:one-end-1} indeed infer that at each point on the singularity $S$, there is a $W^{2,p}$ coordinate system (with respect to the original smooth structure) such that $g$ is flat. By the Sobolev embedding, such coordinate system is $C^{1,\alpha}$. From these, we expect that one should be able to construct the global $C_{\mathrm{loc}}^{1,\a}$ diffeomorphism using the method of Cheeger \cite{Cheeger1970}. In this subsection, we will make use of the Ricci flow approach to construct a $C^{1,\alpha}_{\mathrm{loc}}$ diffeomorphism when $p$ is finite and Theorem \ref{intro-rigidity} when $p=\infty$.

\medskip

Let $h$ be a smooth metric on $M$. A family of metrics $g(t)$ is said to be a Ricci-Deturck flow with background metric $h$ starting from $g_0$ if it satisfies
\begin{equation}
    \left\{
    \begin{array}{ll}
       \partial_t g_{ij}  =-2R_{ij}+\nabla_i W_j +\nabla_j W_i;\\[1.5mm]
       W^k= g^{pq}(\Gamma_{pq}^k-\tilde \Gamma_{pq}^k); \\[1mm]
         g(0)=g_0,
    \end{array}
    \right.
\end{equation}
where $\tilde \nabla$ denotes the connection of $h$. Furthermore, let $\Psi_t$ be the solution to
\begin{equation}\label{ODE-RF}
    \left\{
    \begin{array}{ll}
         \partial_t\Phi_t(x)=-W(\Phi_t(x),t);  \\[1mm]
        \Phi_0(x) =x.
    \end{array}
    \right.
\end{equation}
Then $\hat g(t)=\Phi_t^*g(t)$ is a Ricci flow with $\hat g(0)=g_0$. It is very often to choose the reference metric $h$ so that it is of bounded geometry of infinity order on $M$ in the sense of the following
\begin{defn}\label{Defn: 4.1}
We say that $h$ has bounded geometry of  order $k$ if $\mathrm{inj}(M,h)>0$ and for all $0\leq m\leq k$, there is $C_m>0$ such that $\sup_M |\nabla^m \Rm(h)|\leq C_m$.
\end{defn}
We remark here that by the work of Shi \cite{Shi1989}, any metric with bounded geometry of $0$ order can be perturbed slightly in $C^1$ to another metric with bounded geometry of infinity order.

\begin{defn}\label{Defn: 4.2}
A metric $g$ is said to be $L^\infty(M)\cap W^{1,q}(M)$ with respect to $h$ if $|\nabla^h g|\in L^q(M,h)$ and $\Lambda^{-1}g\leq h\leq \Lambda h$ for some $\Lambda>0$ almost everywhere on $M$. 
\end{defn}

\begin{proof}[Proof of Theorem~\ref{intro-rigidity}]

To avoid confusion when using the flow, we will use $g_0$ instead of $g$ to denote the given metric.

When $n<p<\infty$. If $m(M,g,\mathcal{E})=0$, thanks to Theorem~\ref{thm:one-end-1}, $M$ is one-ended and hence we can find $h$ so that $h$ is smooth, coincides with the Euclidean metric on $\mathcal{E}$ and $g_0$ is $L^\infty(M)\cap W^{1,p}(M)$ with respect to $h$. If $p=\infty$, this follows from the assumption as we might replace $h$ by the Euclidean at $\mathcal{E}$ via a cutoff function. So in either case, we might assume $g_0$ to be $L^\infty(M)\cap W^{1,p}(M)$ with respect to some $h$ on $M$ and $g_0$ is $\e_n$ close to $h$ in $C^0$ where $h$ is of bounded geometry of infinity order for any given $\e_n>0$.

\begin{claim}\label{RFe-claim}
We can construct a Ricci-Deturck flow $g(t),t\in (0,T]$ with respect to metric $h$ for some $T>0$ such that it satisfies
\begin{enumerate}\setlength{\itemsep}{1mm}
\item[(a)] $g(t)\to g_0$ as $t\to 0$ in $C^0_{\mathrm{loc}}(M) \cap C^\infty_{\mathrm{loc}}(M-S)$;
\item[(b)]  there exists $C_0>0$ and $Q=\frac12 (1+n/p)\in (0,1)$ such that 
\begin{equation*}
t^{-\frac{1}{2}}|\tilde \nabla g(t)| +|\tilde \nabla^2 g(t)| \leq \frac{C}{t^Q};
\end{equation*}
\item[(c)] there exists $C_1>0$ such that for all $x\in M$ and $t\in (0,T]$,
\begin{equation*}
    ||\tilde \nabla g(t)||_{L^p(B_g(x,1))} \leq C_1;
\end{equation*}
\item[(d)] $g(t)$ is $2\e_n$ close to $h$ in $C^0$ for $t\in (0,T]$;
\item[(e)] $g(t)$ is asymptotically flat at $\mathcal{E}$.
\end{enumerate}
Here $\tilde \nabla$ denotes the connection of $h$. If $p=\infty$, $Q$ is understood to be $1/2$.
\end{claim}
\begin{proof}[Proof of Claim \ref{RFe-claim}]
The existence follows from \cite{Simon2002} while the improved estimates follows from discussion of \cite[Section 3-4 and Theorem 7.2]{ShiTamPJM}, see also \cite{Simon2002} and \cite{McFeronSzekelyhidi2012} for the discussion in case of $p=\infty$.
\end{proof}

It suffices to show that $g(t)$ is Ricci flat. If this is the case, then $(M,g(t))$ is isometric to the Euclidean space by the volume comparison. Moreover, if we consider the ODE:
\begin{equation}
\left\{
\begin{array}{ll}
\partial_t \Psi_t(x)=  W( \Psi_t(x),t);\\[1mm]
\Psi_T(x)=x,
\end{array}
\right.
\end{equation}
for $(x,t)\in M\times (0,T]$, then the Ricci flatness implies 
\begin{equation}
\partial_t\left[ (\Psi_t^{-1})^*  g(t)\right]=0.
\end{equation}
Hence,  $ g(t)=\Psi_t^*  g(T)$ where $ g(T)$ is isometric to a standard Euclidean metric.  By \cite[(5.2)]{CalabiHartman1970},
\begin{equation}\label{local-DIFF-CalabiHartman}
\frac{\partial^2 \Psi_t^m}{\partial x^i \partial x^j} = \Gamma^k_{ij}(g(t)) \frac{\partial \Psi_t^m}{\partial x^k} -\Gamma( g(T))_{kl}^m \frac{\partial \Psi_t^l}{\partial x^i}\frac{\partial \Psi_t^k}{\partial x^j}
\end{equation}
in local coordinate of $M$.  Thanks to (c) in Claim \ref{RFe-claim} and metrics equivalence,  the right hand side of \eqref{local-DIFF-CalabiHartman} is bounded in $L^p_{\mathrm{loc}}(M)$ for $p>n$ as $t\to 0$ and hence $\Psi_t\in W_{\mathrm{loc}}^{2,p}(M)$ uniformly in $t\to 0$. By the Sobolev embedding,  we may pass $\Psi_t\to \Psi_0$ for $t\to 0$ in $C_{\mathrm{loc}}^{1,\a}(M)$ for some $\a>0$. Moreover, since $g(t)$ is uniformly equivalent to $g$ and
$$d_{g(t)}(x,y)=d_{g(T)}(\Psi_t(x),\Psi_t(y))$$
for all $x,y\in M$ and $t\in (0,T]$. We see that $\Psi_0$ is a bi-Lipschitz map and hence is a diffeomorphism. This will complete the proof when $n<p\leq\infty$ since $ g(t)\to g_0$ in $C^0_{\mathrm{loc}}(M)$ as $t\to 0$.

\medskip

The remaining part of the proof is devoted to prove $\Ric(g(t))\equiv0$. We first consider the case of $p=\infty$. The argument is analogous to that of \cite[Theorem 1.1]{LeeTam2021}. By Proposition~\ref{Prop:ricci-flat}, we know that $\Ric(g_0)=0$ outside $S$. Let $\varphi=|\Ric(g(t))|$.
Fix $x_0\in S$ and $\rho(x)$ be a distance like function on $M$ such that $\rho$ is $2$ bi-Lipschitz to $d_h(x,x_0)$ with $|\tilde\nabla \rho|^2+ |\tilde\nabla^2\rho|\leq C$ obtained by \cite{Tam}. let $\Phi=\phi^m(r^{-1}\rho)$ be a cutoff function on $M$ such that $\Phi=1$ on $B_h(x_0,r)$ and vanishes outside $B_h(x_0,2r)$ for $r$ sufficiently large. It suffices to show that as $r\to +\infty$,
\begin{equation}
E_r(t)=\int_M \varphi \, \Phi \,d\mu_{g(t)}=o(1).
\end{equation}

Recall that $(\partial_t-\Delta_{t}) \Ric=\Ric *\Rm$ along the Ricci flow and hence $(\partial_t-\Delta_{g(t)}) \varphi \leq Ct^{-Q}\varphi+\langle W,\nabla \varphi\rangle$ in the sense of distribution along the Ricci-Deturck flow $g(t)$ where $W$ is defined by \eqref{ODE-RF}. Direct computation using Claim~\ref{RFe-claim} shows that
\begin{equation}
    \begin{split}
        \partial_tE_r
        &\leq Ct^{-Q} E_r+ Ct^{-Q(1+\frac1m)}r^{-1} E_r^{1-\frac1m}.
    \end{split}
\end{equation}
We fix $m$ large enough so that $Q(1+\frac1m)<1$. By integrating from $0<s\ll1$, it suffices to control $\lim_{t\to 0^+} E_r(t)$. We compute as follows: for $t\to 0^+$,
\begin{equation}
E_r(t)\leq \left(\int_{B_h(x_0,2r) \cap S_{\Lambda \sqrt{t}}} +\int_{B_h(x_0,2r) \setminus  S_{\Lambda \sqrt{t}}} \right) \varphi\; d\mu_{g(t)}
= \mathbf{I}+\mathbf{II}.
\end{equation}
We will choose $\Lambda$ sufficiently large such that \cite[Proposition 3.1]{LeeTam2021} applied to conclude for some $l$ sufficiently large (uniformly in $r$ and $t$), we have as $t\to 0$,
\begin{equation}
    \begin{split}
        \mathbf{II}&\leq o_r(1)\cdot \int_{\Lambda\sqrt{t}}^{2r} t^l  r^{-2(l+1)+1}\leq o_r(1) + Cr^{-1}.
    \end{split}
\end{equation}
Here we have used the assumption on $\mathcal{H}^{n-1}(S)=0$ so that for any compact set $\Omega$, $\mathcal{H}^{n-1}(\Omega\cap S_{r'})=o_\Omega(1)\cdot r'$ as $r'\to 0$.

For $\mathbf{I}$, we use the estimate in Claim~\ref{RFe-claim} to see that as $t\to 0$,
\begin{equation}
    \begin{split}
        \mathbf{I}&\leq C\mu_{h}(B_h(x_0,2r)\cap S_{\Lambda\sqrt{t}}) \cdot t^{-Q}=o_r(1).
    \end{split}
\end{equation}

Combining the estimates of $\mathbf{I}$ and $\mathbf{II}$, we deduce that $\lim_{t\to 0} E_r(t)\leq Cr^{-1}$ for some uniform constant $C>0$ and hence $E_\infty=0$ by letting $r\to +\infty$. This completes the proof when $p=\infty$.

If $n<p<\infty$, \cite[Lemma 2.7]{JiangShengZhang2020} implies that $R(g_0)\geq 0$ in the distributional sense. Since Theorem~\ref{thm:one-end-1} implies that $M$ is topologically Euclidean and $S$ is bounded, a slight modification of \cite[Theorem 1.1]{JiangShengZhang2021} implies that $g(t)$ is of $R(g(t))\geq 0$ in the classical sense. We remark here that since $S$ is bounded and $g_0$ is smooth outside $S$ with $R(g_0)\geq 0$, it suffices to consider the distributional scalar lower bound on a compact region nearby $S$ and hence the argument in \cite[Theorem 1]{JiangShengZhang2021} can be carried over easily. Once it is known that $R(g(t))\geq 0$, we can easily deduce that $g(t)$ is flat by the rigidity in the classical positive mass theorem using the argument in \cite{ShiTamPJM,McFeronSzekelyhidi2012}. This proves the case of $p\in (n,+\infty)$ and hence completes the proof. 
\end{proof}

\end{document}